\documentclass[12pt,letterpaper]{article}

\usepackage{geometry}
\usepackage{verbatim}

\usepackage{t1enc}
\usepackage[english]{babel}

\usepackage{amsmath,amssymb,amsthm}
\usepackage{mathrsfs,esint}
\usepackage{graphicx,wrapfig}
\usepackage[format=hang]{caption}

\newcommand*{\figfilename}{PosCurvature-Figures.pdf}
\newcommand*{\Ccal}{C} 
\newcommand*{\Hcal}{{\mathcal H}}
\newcommand*{\C}{{\mathbb C}}
\newcommand*{\R}{{\mathbb R}}
\newcommand*{\Leb}[1]{{\mathscr L}^{#1}}
\newcommand*{\N}{{\mathbb N}}
\newcommand*{\loc}{{\mathrm{loc}}}
\newcommand*{\eps}{\varepsilon}
\newcommand*{\BV}{\mathrm{BV}}
\newcommand*{\Tr}{T}
\newcommand*{\ssub}{\sqsubset}
\newcommand*{\Om}{\Omega}
\newcommand*{\dOm}{\partial\Omega}
\newcommand*{\symdiff}{\bigtriangleup}
\providecommand*{\coloneq}{:=}
\providecommand*{\eqcolon}{=:}
\DeclareMathOperator{\Lip}{Lip}
\DeclareMathOperator{\LIP}{LIP}
\DeclareMathOperator{\Ext}{Ext}
\DeclareMathOperator{\Div}{div}
\DeclareMathOperator{\dist}{dist}
\DeclareMathOperator{\diam}{diam}
\DeclareMathOperator{\supt}{supt}

\numberwithin{equation}{section}
\theoremstyle{plain}
\newtheorem{thm}[equation]{Theorem}
\newtheorem{prop}[equation]{Proposition}
\newtheorem{cor}[equation]{Corollary}
\newtheorem{lem}[equation]{Lemma}

\theoremstyle{definition}
\newtheorem*{quest}{Question}
\newtheorem{defn}[equation]{Definition}
\newtheorem{remark}[equation]{Remark}
\newtheorem{exa}[equation]{Example}
\begin{document}
\title{Domains in metric measure spaces with boundary of positive mean curvature, and the Dirichlet problem for functions of least gradient}
\author{Panu Lahti\and Luk\'a\v{s} Mal\'y\and Nageswari Shanmugalingam\and Gareth Speight%
\thanks{The authors thank Estibalitz Durand-Cartagena, Marie Snipes and
Manuel Ritor\'e for fruitful discussions about the subject of the paper.
The research of N.S.\@ was partially supported by the NSF grant~\#DMS-1500440 (U.S.).
The research of P.L.\@ was supported by the Finnish Cultural Foundation.
The research of L.M.\@ was supported by the Knut and Alice Wallenberg Foundation (Sweden).
Part of this research was conducted during the visit 
of N.S.\@ and P.L.\@ to Link\"oping University. The authors wish to thank this institution for its kind hospitality.}}
\maketitle
\begin{abstract}
We study the geometry of domains in complete metric measure spaces equipped with a doubling measure 
supporting a $1$-Poincar\'e inequality. We propose a notion of \emph{domain with boundary of positive 
mean curvature} and prove that, for such domains, there is always a solution to the Dirichlet problem for 
least gradients with continuous boundary data. Here \emph{least gradient} is defined as minimizing total 
variation (in the sense of BV functions) and boundary conditions are satisfied in the sense that the 
\emph{boundary trace} of the solution exists and agrees with the given boundary data. This 
extends the result of Sternberg, Williams and Ziemer~\cite{SWZ} to the non-smooth setting. Via counterexamples 
we also show that uniqueness of solutions and existence of \emph{continuous} solutions can fail, even 
in the weighted Euclidean setting with Lipschitz weights.  
\end{abstract}
\section{Introduction}
The work of Giusti~\cite{Gi} showed a close connection between the curvature of the boundary of a Euclidean domain $\Om\subset\R^n$ and the existence of a solution to the Dirichlet problem related to the Plateau problem
\[
\Div(\nabla u(1+|\nabla u|^2)^{−1/2})=H(x,u(x))
 \quad\text{ for }\Leb{n}\text{-a.e.\@ } x\in \Om,
\]
with the graph of $u$ a minimal surface (under the constraint that it has prescribed mean curvature) in $\R^{n+1}$. 
The work of Barozzi and Massari~\cite{BM} studied a related obstacle problem for BV energy minimizers,
where the obstacle is required to have a certain curvature condition analogous to that of~\cite{SWZ}. While the
conditions in~\cite{SWZ} only considered domains whose boundary is of non-negative (or positive) mean
curvature, the paper~\cite{BM} imposed a more general mean curvature condition on the obstacle $M$, namely, that
if $g\in L^1_{\loc}(\Om)$ and $M\subset \Om$, then $M$ is of mean curvature at most $g$ if 
\[
 P(M, \Om')\le P(F,\Om')+\int_{M\setminus F}g
\]
whenever $\Om'\Subset\Om$ and $F\symdiff M\Subset\Om'$. The notion of non-negative
mean curvature of $\partial\Om$ (for $\Om$ whose boundary need not be smooth), as given in~\cite{SWZ} is
not quite this condition, but is similar.
Following this, the work
of~\cite{SWZ} showed that the Dirichlet problem related to the least gradient problem
\[
\text{div}\frac{\nabla u}{|\nabla u|}=0\text{ in }\Om, \quad T u=f\text{ on }\dOm,
\]
has a solution if and only if $\dOm$ has non-negative mean curvature (with respect to the
domain $\Om$) and $\dOm$ is nowhere locally area-minimizing. Here $Tu$ is the
trace of $u$ to $\dOm$ (see the next section for its definition and discussion regarding
its existence). More general notions of Dirichlet problem such as minimizing the
energy integral
\[
 I(u)=\| Du\|(\overline{\Om})
\]
and the energy functional
\begin{equation}\label{eq:p-to-1-Jfunct}
 J(u)=\| Du\|(\Om)+\int_{\dOm}|Tu-f|\, d\Hcal
\end{equation}
over all BV functions $u$ on $\R^n$ (with $u=f$ on $\R^n\setminus\Om$ for the
energy $I$) were studied for
more general Euclidean domains, for example, in~\cite{BDaM}, see also the discussions
in~\cite{Gi2, AG, SWZ, MRS}. Should we obtain a BV energy minimizer on $\Om$ with the
correct trace $f$ on $\dOm$, then this solution also minimizes $I$ and $J$. Until
the work of Sternberg, Williams, and Ziemer~\cite{SWZ}, not much consideration was
given to how the trace of the minimizers fit in with the boundary data.

The recent development of first order analysis in the metric setting (see~\cite{HKST})
led to the extension of the theory of BV functions and functions of least gradient in
the metric setting, see~\cite{Mr, A, AdM, AMP, KKLS, HKLS} for a sample. The
papers~\cite{HKLS, KLLS} studied minimizers of the energy functionals $I$ and $J$
in the metric setting. The goal of the present paper is to study existence of the
strongest possible solutions to the Dirichlet problem in the metric setting, namely
that the solution obtains the correct prescribed trace value on the boundary of the domain
of interest.

In addition to Euclidean domains as mentioned above, curvature conditions for the boundary of the domain
also show up in the Heisenberg setting.
Extending some of the results regarding the Plateau problem from the
Euclidean space (see for example~\cite{GS, Gi2}) to the Heisenberg setting,
the recent paper~\cite{PSTV} studied a related minimization problem in the
Heisenberg setting, and there too it seems curvature of the boundary plays a role. More specifically,
in~\cite{PSTV} it is shown that if $\Om\subset\R^{2n}$ is a bounded Lipschitz domain and
$\varphi$ is a Lipschitz function on $\dOm$ that is affine on the parts of the boundary where the domain is not
positively curved, then there is a function $u$ on $\Om$ such that the subgraph of $u$ in $\R^{2n}\times\R$,
equipped with the Heisenberg metric, is of minimal boundary surface with the trace of $u$ on $\dOm$ equal to
$\varphi$, and furthermore, $u$ is Lipschitz continuous on $\Om$. The work of~\cite{PSTV} therefore
is also concerned with the minimal graph problem rather than the least gradient problem. Any discussion
of the Dirichlet problem for least gradient functions on domains \emph{in} the Heisenberg group itself
should be governed by curvature of the boundary of the domain as well.

We propose an analog (Definition~\ref{posmeancur}) of the notion of positive mean curvature from the weak
formulation of~\cite{SWZ} to the metric setting where the measure is doubling and supports
a $1$-Poincar\'e inequality. 
The main theorem of this paper, Theorem~\ref{thm:main}, will demonstrate the existence
of such a strong solution to the least gradient problem for (globally) continuous BV boundary data provided the boundary
of the domain is of positive mean curvature in the sense considered here. We will also
show in the last section of this paper that outside of the Euclidean setting,
continuity (inside the domain) of the solution and uniqueness of the solution can
fail; indeed, the examples we provide can easily be modified to be a domain
in a Riemannian manifold. We point out that our definition of positive mean curvature of the boundary
is somewhat different from that of~\cite{SWZ}, see Remark~\ref{rem:curv-diff-defn} below.

The focus of~\cite{SWZ} was Lipschitz boundary data; for such data, the authors prove that the
solutions obtained are also Lipschitz (up to the boundary). The examples we provide here
show that even with Lipschitz boundary data, Lipschitz continuity of the solution is not guaranteed
in the general setting (not even in the Riemannian setting). Therefore, we broaden
our scope to the wider class of all globally continuous BV functions as boundary data.
We also show that if $\Om$ satisfies some additional conditions, then
it suffices to know that the boundary datum $f$ is merely a continuous function on
the boundary $\dOm$, see Section~\ref{sec:cont-data}.

The primary tool developed
in the present paper, ``stacking pancakes with minimal boundary surface'',
uses the idea that superlevel sets of functions of least gradient are of minimal boundary
surface (in the sense of~\cite{KKLS}). In the Euclidean setting, this was first proven
by Bombieri, De Giorgi, and Giusti in~\cite{BGG}, and was used in that spirit
in the work~\cite{SWZ}, which inspired our work presented here. In the
metric setting this minimality of the layers, or superlevel sets, was proven
in~\cite{HKLS}. While this method of ``stacking pancakes'' is similar to the one in~\cite{SWZ},
the tools available to us in our setting are very limited. In particular, we do not have the
smoothness properties and tangent cones for boundaries of sets of minimal boundary surfaces,
 and hence the construction of ``pancakes'' (superlevel sets) given in~\cite{SWZ} is not permitted
 to us. Furthermore, in the Euclidean setting, it is shown in~\cite{SWZ} that if two sets $E_1,E_2$ of minimal
 boundary surface such that $E_1\subset E_2$ have intersecting boundaries, then the two boundaries coincide
 in a relatively open set, and hence it holds that $E_1=E_2$. This property is used to show that the function
 constructed from the ``stack of pancakes'' is necessarily continuous, and thus issues of measurability of the
 constructed function does not arise in the Euclidean setting of~\cite{SWZ}.
 In the metric setting this property fails (see the
 examples constructed in the final section of this paper). Consequently we had to modify our construction of
 the solution function from the ``stack of pancakes'' by considering a countable sub-stack of pancakes.
\section{Notations and definitions in metric setting}
\label{sec:metric-setting}
We will assume throughout the paper that $(X,d,\mu)$ is a complete metric
space endowed with a doubling 
measure $\mu$ that satisfies a $1$-Poincar\'e inequality defined below.
We say that the measure $\mu$ is \emph{doubling} on $X$
if there is a constant $C_D\ge 1$ such that
\[
0<\mu(B(x,2r))\le C_D\, \mu(B(x,r))<\infty
\]
whenever $x\in X$ and $r>0$. Here $B(x,r)$ denotes the open ball with center $x$ and radius $r$.
Given measurable sets $E,F \subset X$, the symbol $E\ssub F$ will denote that 
$\mu(E\setminus F)=0$ or, in other words, $\chi_E \le \chi_F$ $\mu$-a.e.

A complete metric space with a doubling measure is proper,
that is, closed and bounded sets are compact.
Since $X$ is proper, given an open set $\Om\subset X$
we define $L^1_{\loc}(\Om)$ to be the space of
functions that are in $L^1(\Om')$ for every $\Omega'\Subset \Om$, that is,
when the closure of $\Om'$ is a compact subset of $\Om$.  
Other local spaces of functions are defined analogously.

Given a function $u:X\to\overline{\R}$, we say that a non-negative Borel-measurable function $g$ is an \emph{upper gradient} of $u$ if
\begin{equation}\label{eq:upperGr}
|u(y)-u(x)|\le \int_\gamma g\, ds
\end{equation}
whenever $\gamma$ is a non-constant compact rectifiable curve in $X$. The endpoints of $\gamma$ are denoted by
$x$ and $y$ in the above inequality. The inequality should be interpreted to mean that $\int_\gamma g\, ds=\infty$
if at least one of $u(x)$, $u(y)$ is not finite.

We say that $X$ supports a {\em $1$-Poincar\'e inequality} if there are positive constants $C_P$, $\lambda$
such that
\[
 \fint_B|u-u_B|\, d\mu\le C_P\, r\, \fint_{\lambda B}g\, d\mu
\]
whenever $B=B(x,r)$ is a ball in $X$ and $g$ is an upper gradient of $u$.
Here, $u_B\coloneq\mu(B)^{-1}\int_B u\, d\mu\eqcolon \fint_B u\, d\mu$ is the average of $u$ on the ball $B$, and $\lambda B\coloneq B(x,\lambda r)$.

Throughout this paper $C$ will denote a constant whose precise value is not of interest here and depends solely on $C_D$, $C_P$, $\lambda$, and perhaps on the domain $\Om$. As $C$ stands
for such a 
generic constant, its value could differ at each occurrence.

Let $\widetilde{N}^{1,1}(X)$ be the
class of all $L^1$ functions on $X$ for which there
exists an upper gradient in $L^1(X)$. For
$u\in\widetilde{N}^{1,1}(X)$ we define
$$
\|u\|_{\widetilde{N}^{1,1}(X)}=\|u\|_{L^1(X)}+\inf_{g}\|g\|_{L^1(X)},
$$
where the infimum is taken over all upper gradients $g$ of
$u$. Now, we define an equivalence
relation in $\widetilde{N}^{1,1}(X)$ by $u\sim v$ if and only if
$\|u-v\|_{\widetilde{N}^{1,1}(X)}=0$.

The \emph{Newtonian space} $N^{1,1}(X)$
is defined as the quotient $\widetilde{N}^{1,1}(X)/\sim$ and
it is equipped with the norm
$\|u\|_{N^{1,1}(X)}=\|u\|_{\widetilde{N}^{1,1}(X)}.$ One can define analogously $N^{1,1}(\Omega)$ for an open set $\Omega\subset X$. For more on upper gradients and Newtonian spaces
of functions on metric measure spaces, see~\cite{HKST}.

For $u\in L^1_{\loc}(X)$ the {\em total variation} of $u$ is defined by
\[
\|Du\|(X) = \inf\left\{\liminf_{i\to\infty}\int_X g_{u_i}\,d\mu:N^{1,1}_{\loc}(X)\ni u_i\to u\,\text{in }L^1_{\loc}(X)\right\},
\]
where $g_{u_{i}}$ are upper gradients of $u_{i}$.

One can define analogously $\|Du\|(\Omega)$ for an open set $\Omega\subset X$.
If $A\subset X$ is an arbitrary set we define
\[
\|Du\|(A)=\inf\{\|Du\|(\Omega): \Omega\supset A, \Omega\subset X \,\text{open}\}.
\]
A function $u\in L^1(X)$ is in $\BV(X)$ ({\em of bounded variation}) if $\|Du\|(X)<\infty$.
For such $u$, $\| Du\|$ is a Radon measure on $X$, see
\cite[Theorem 3.4]{Mr}.
A $\mu$-measurable set $E\subset X$ is of {\em finite perimeter} if $\|D\chi_E\|(X)<\infty$. The {\em perimeter} of $E$ in $\Omega$ is
\[
P(E,\Omega)\coloneq \|D\chi_E\|(\Omega).
\]
BV energy on open sets is lower semicontinuous with respect to $L^1$-convergence, i.e., if $u_k \to u$ in $L^1_\loc(\Om)$ as $k\to\infty$, where $\Om \subset X$ is open, then
\begin{equation}
  \label{eq:BV-lsc}
  \|Du\|(\Om) \le \liminf_{k\to\infty} \|Du_k\|(\Om)\,.
\end{equation}
The coarea formula in the metric setting~\cite[Proposition~4.2]{Mr} says
that if $u\in L^1_{\loc}(\Om)$ for an open set $\Om$, then
\begin{equation}
  \label{eq:coarea-formula}
\|Du\|(\Om)=\int_{-\infty}^{\infty} P(\{u>t\},\Om)\,dt,
\end{equation}
If $\| Du\|(\Om)<\infty$, the above holds with $\Om$ replaced by any Borel set
$A\subset \Om$.

Given a set $E\subset X$, its {\em Hausdorff measure of codimension 1} is defined by
\[
\Hcal(E)=\lim_{r\to 0}\
\inf\left\{\sum_{i=1}^{\infty}\frac{\mu(B(x_i,r_i))}{r_i}:E\subset\bigcup_{i=1}^{\infty} B(x_i,r_i),
r_i\le r\right\}.
\]
 It is known from~\cite[Theorem~5.3]{A} 
and~\cite[Theorem~4.6]{AMP} that if $E\subset X$ is of finite
 perimeter, then for Borel sets $A\subset X$,
 \[
 \frac{1}{C}\Hcal(A\cap\partial_mE) \le P(E,A) \le C\Hcal(A\cap\partial_mE),
 \]
 where
 $\partial_mE$ is the measure-theoretic boundary of $E$, that is, the collection
 of all points $x\in X$ for which simultaneously
 \[
 \limsup_{r\to 0^+}\frac{\mu(B(x,r)\cap E)}{\mu(B(x,r))}>0
 \quad
 \text{and}
 \quad
 \limsup_{r\to 0^+}\frac{\mu(B(x,r)\setminus E)}{\mu(B(x,r))}>0.
 \]

 Given a bounded domain $\Omega\subset X$ and a function $u\in \BV(\Omega)$, we say that
 $u$ has a trace at a point $z\in\partial\Omega$ if there is a number $\Tr u(z)\in\R$
 such that
\begin{equation}
  \label{eq:trace-def}
  \lim_{r\to 0^+}
    \fint_{B(z,r)\cap\Omega}|u(x)-\Tr u(z)|\, d\mu(x) = 0.
 \end{equation}
We know from~\cite[Theorem 3.4, Theorem 5.5]{LS} and
\cite{MSS} that if $\Omega$  satisfies all of the following geometric
conditions, then every function in $\BV(\Omega)$ has a trace $\mathcal H$-a.e. on $\partial\Omega$:
\begin{enumerate}
\item \label{enum:conditions for trace} there is a constant $C\ge 1$ such that
\[
  \mu(B(z,r)\cap\Omega) \ge \frac{\mu(B(z,r))}{C}
\]
whenever $z\in\overline{\Omega}$ and $0<r<2 \diam(\Omega)$;
\item there is a constant $C\ge 1$ such that
\[
 \frac{1}{C} \frac{\mu(B(z,r))}{r} \le \Hcal(B(z,r)\cap\partial\Omega) \le C\, \frac{\mu(B(z,r))}{r}
\]
whenever $z\in\partial\Omega$ and $0<r<2 \diam(\Omega)$;
\item  $\Omega$ supports a $1$-Poincar\'e inequality.
\end{enumerate}
Furthermore, if $\Omega$ satisfies all the above conditions, then the trace class of
$\BV(\Omega)$ is $L^1(\partial\Omega,\Hcal)$.
\begin{defn}\label{def: least Grad}
Let $\Om\subset X$ be an open set, and let $u\in \BV_{\loc}(\Om)$. We say that $u$ is of \emph{least gradient in $\Om$} if
\[
\| Du\|(V)\le \| Dv\|(V)
\]
whenever $v\in \BV(\Om)$ with $\overline{\{x\in\Om\, :\, u(x)\ne v(x)\}}\subset V\Subset\Om$.
A set $E \subset \Om$ is of \emph{minimal boundary surface in $\Om$}, if $\chi_E$ is of least gradient in $\Om$.
\end{defn}
\begin{defn}\label{def:weak-strong-solutions}
 Let $\Omega$ be a nonempty bounded domain in $X$ with $\mu(X\setminus\Omega)>0$, and let
 $f\in \BV_{\loc}(X)$. We say that $u\in \BV_{\loc}(X)$ is a
\emph{weak solution to the Dirichlet problem for least gradients in $\Om$ with
boundary data $f$}, or simply, \emph{weak solution to the Dirichlet problem with boundary
data $f$}, if $u=f$ on $X\setminus\Om$ and
\[
\| Du\|(\overline{\Om})\le \| Dv\|(\overline{\Om})
\]
whenever $v\in \BV(X)$ with $v=f$ on $X\setminus\Om$.
\end{defn}
\begin{defn}\label{def:strong-solutions}
Let $\Omega$ be a nonempty domain in $X$ and $f:\partial\Omega\to\R$. 
 We say that a 
 function $u\in \BV(\Omega)$ is a \emph{solution to the
 Dirichlet problem for least gradients in $\Om$ with boundary data $f$}, 
or simply, \emph{solution to the
 Dirichlet problem with boundary data $f$,} if $\Tr u=f$ $\mathcal H$-a.e.
 on $\partial\Omega$ and whenever
 $v\in \BV(\Omega)$ with $\Tr v=f$ $\mathcal H$-a.e. on $\partial\Omega$ we must have
 \[
 \| Du\|(\Omega)\le \| Dv\|(\Omega).
 \]
 \end{defn}

Note that solutions and weak solutions to Dirichlet problems on a domain $\Om$ are
necessarily of least gradient in $\Om$.

Given a function $u$ on $X$ and $x\in X$, we define
\[
 u^\vee(x) = \mathop{\textup{ap-$\limsup_{y\to x}$}} u(y) \coloneq
   \inf\biggl\{ t\in\R\colon \lim_{r\to 0^+}\frac{\mu(B(x,r)\cap\{u>t\})}{\mu(B(x,r))}=0\biggr\},
\]
and
\[
 u^\wedge(x) = \mathop{\textup{ap-$\liminf_{y\to x}$}} u(y) \coloneq
   \sup\biggl\{ t\in\R\colon \lim_{r\to 0^+}\frac{\mu(B(x,r)\cap\{u<t\})}{\mu(B(x,r))}=0\biggr\}.
\]
Then, $u^\vee(x)=u^\wedge(x)$ for $\mu$-a.e.\@ $x\in X$ by the Lebesgue differentiation theorem provided that $u\in L^1_\loc(X)$.

Points $x$
for which $u^\vee(x)=u^\wedge(x)$ are said to be \emph{points of approximate continuity} of $u$.
Let $S_{u}$ be the set of points $x$ at which $u$ is not approximately continuous.
For $u\in \BV(X)$, the set $S_u$ is of $\sigma$-finite codimension 1 Hausdorff measure,
see~\cite[Proposition 5.2]{AMP}.
If in addition $u=\chi_E$ for some $E\subset X$, then $S_u=\partial_m E$.
By~\cite[Theorem 5.3]{AMP}, the Radon measure $\| Du\|$ associated with a
function $u\in \BV(X)$ permits the
following decomposition:
\begin{equation}\label{eq:variation measure decomposition}
d\| Du\|=g\, d\mu+d\| D^ju\|+d\| D^cu\|,
\end{equation}
where $g\, d\mu$ with $g\in L^1(X)$ gives the part of $\| Du\|$ that
is absolutely continuous with
respect to the underlying measure $\mu$ on $X$, and $\| D^ju\|$ is the
so-called jump-part of $u$. This latter measure lives inside $S_u$, and is absolutely
continuous with respect to $\Hcal\lfloor_{S_u}$.
The third measure, $\| D^cu\|$,
is called the Cantor part of $\| Du\|$,
and does not charge sets of $\sigma$-finite codimension 1 Hausdorff measure.
In the literature, the set $S_u$ is called the \emph{jump set} of $u$, see~\cite{A, AdM, AMP}.

It was shown in~\cite{HKLS} that functions of least gradient, after a modification on a set of measure zero,
are continuous everywhere outside their jump sets. 
\section{Preliminary results related to weak solutions}
\label{sec:prelresults}
Throughout the rest of this paper, we will assume that $X$ is a complete metric space equipped with
a doubling measure $\mu$ supporting a $1$-Poincar\'e inequality, and $\Om\subset X$ is a
nonempty bounded domain such that
$\mu(X\setminus\Om)>0$.

We will need the next lemma for functions of the form $f=\chi_F$ for sets $F\subset X$
of finite perimeter.
\begin{lem} \label{lem:existence-weak-sol}
For every $f\in \BV_\loc(X)$ such that $\Vert Df\Vert(X)<\infty$
there is a function $u_f\in\BV_\loc(X)$ that is a weak solution to the Dirichlet problem
in $\Om$ with boundary data $f$.
\end{lem}
\begin{proof}
Let 
\[
I \coloneq \inf\{\| Du\|(\overline{\Om})\, :\, u\in \BV_\loc(X)\text{ and }u=f\text{ on }X\setminus\Om\}. 
\]
Observe that $0 \le I \le \| Df\|(\overline{\Om})<\infty$. Let $\{u_k\}_{k=1}^\infty$ be a sequence of functions
in $\BV_\loc(X)$ with $u_k=f$ on $X\setminus\Om$ such that $\| Du_k\|(\overline{\Om})\to I$ as $k\to\infty$.
Let $B \subset X$ be an open ball that contains $\overline\Om$. In particular, we can choose $B$ so that $\mu(B \setminus\Om)>0$.
Hence, 
the $1$-Poincar\'e inequality yields that
\begin{align*}
\int_B |u_k-f|\, d\mu\le C_{B,\Om} \| D(u_k-f)\|(\overline{\Om})
  &\le C_{B,\Om}\bigl(\| Du_k\|(\overline{\Om})+\| Df\|(\overline{\Om})\bigr)\\
   &\le 3C_{B,\Om} \| Df\|(\overline{\Om})
\end{align*}
for sufficiently large $k$. Note that the above holds true without subtracting $(u_k-f)_B$ on
the left-hand side because
$u_k-f=0$ on $B\setminus\Om$, while $\mu(B\setminus\Om)$ is positive, see for example~\cite[Lemma~2.2]{KKLS}.
Thus, the sequence $\{u_k-f\}_{k=1}^\infty$ is bounded in $\BV(B)$, and hence so is $\{u_k\}_{k=1}^\infty$. By the 
$1$-Poincar\'e inequality and the doubling property of $\mu$, the space $\BV(B)$ is compactly embedded in $L^{q}(B)$ for 
some $q>1$, see for example~\cite{HaKo, HKST} and~\cite[Theorem 3.7]{Mr}. Therefore,
there is a subsequence, also denoted $u_k$, that converges
in $L^q(B)$ and pointwise $\mu$-a.e.~in $B$ to a function $u_0\in \BV(B)$. By the fact that each $u_k=f$ on
$X\setminus\Om$, we have that $u_0=f$ on $B\setminus\Om$, and that the extension of $u_0$ by $f$ to $X\setminus B$
yields a function in $\BV_\loc(X)$. We denote this extended function by $u_f$. 

Finally, note by the lower semicontinuity of BV energy that
\[
\| Du_f\|(\overline{\Om})+\| Df\|(B\setminus\overline{\Om})=
\| Du_f\|(B)\le \liminf_{k\to\infty}\| Du_k\|(B)=I+\| Df\|(B\setminus\overline{\Om}),
\]
that is, $\| Du_f\|(\overline{\Om})\le I$. Since $u_f=f$ on $X\setminus\Om$ and $u_f\in\BV_\loc(X)$, it follows that
$\| Du_f\|(\overline{\Om})=I$. This completes the proof of the lemma.
\end{proof}

In the following lemma, we will see that the Dirichlet problem in $\Om$ with boundary data $\chi_F$ for some
set $F\subset X$ of finite perimeter has a weak solution given as a function $\chi_E$ for some set
$E \subset \Om \cup F$. Such a set $E$ will be called a \emph{weak solution set}.
\begin{lem}\label{lem:soln-vs-minimalSet}
Let $F\subset X$ with $P(F,X)<\infty$.
Then, there is a set $E\subset X$ with $P(E,X)<\infty$ such that
$\chi_E$ is a weak solution to the Dirichlet problem in $\Om$ with boundary data $\chi_F$.

Moreover, if $u_{\chi_F}$ is a weak solution to the Dirichlet problem with boundary data $\chi_F$, then we can pick any 
$t_0\in (0,1]$ and choose $E$ to be the set
\[
 E_{t_0}=\{x\in X \colon  u_{\chi_F}(x) \ge t_0\}.
\]
\end{lem}
\begin{proof}
By Lemma~\ref{lem:existence-weak-sol}, there is a weak solution $u_{\chi_F}$.
Note that $0\le u_{\chi_F}\le 1$ on $X$ by the maximum principle proven
in~\cite[Theorem 5.1]{HKLS}.

For $t\in (0,1]$, let
\[
  E_t= \{x\in X\, \colon\,  u_{\chi_F}(x)\ge t\}.
\]
We will first show that $\chi_{E_t}$ is a weak solution to the Dirichlet
problem in $\Om$ with boundary data $\chi_F$ for all $t\in (0,1] \setminus N$ for some negligible set $N$. 
We will prove that $N = \emptyset$ later.

The coarea formula \eqref{eq:coarea-formula},
together with the fact that $P(F,X)<\infty$, gives that
\[
\int_0^1P(E_t,X)\, dt = \int_{-\infty}^\infty P(E_t,X)\, dt = \| Du_{\chi_F}\|(X)\le P(F,X)< \infty,
\]
whence $P(E_{t},X)<\infty$ for $\Leb1$-a.e.\@ $t \in (0,1]$. Moreover, $\chi_{E_t}=\chi_F$ on $X\setminus\Om$ 
for every $t\in(0,1]$. Since $u_{\chi_F}$ is a weak solution corresponding to the boundary data $\chi_F$, we have
$\| Du_{\chi_F}\|(\overline{\Om})\le P(E_t,\overline{\Om})$ for every $t\in(0,1]$.

Let $N = \{t \in (0,1] \colon \|Du_{\chi_F}\|(\overline{\Om}) < P(E_t,\overline{\Om})\}$. Then
by the coarea formula,
\begin{align*}
  \|Du_{\chi_F}\|(\overline{\Om}) = \int_0^1 P(E_t,\overline{\Om})\, dt & = \int_{(0,1] \setminus N} \|Du_{\chi_F}\|(\overline{\Om})\,dt + \int_N P(E_t,\overline{\Om})\, dt \\
  & = (1-\Leb1(N)) \|Du_{\chi_F}\|(\overline{\Om}) + \int_N P(E_t,\overline{\Om})\, dt.
\end{align*}
Hence, $\Leb1(N) \|Du_{\chi_F}\|(\overline{\Om}) = \int_N P(E_t,\overline{\Om})\, dt$, which can hold true only if $\Leb1(N)=0$.

We have shown that $\chi_{E_t} = \chi_F$ on $X\setminus \Om$ and $P(E_t,\overline{\Om}) = \| Du_{\chi_F}\|(\overline{\Om})$
 for every $t\in (0 ,1]\setminus N$, where $\Leb1(N)=0$. Therefore,
for every $t\in (0,1]\setminus N$,
the function $\chi_{E_t}$ is a weak solution with boundary data $\chi_F$, and we may choose $E$ to be the set $E_{t}$.

Let us now show that $N$ is in fact empty. Indeed, taking an arbitrary $t\in (0, 1]$ and a sequence $t_k\in(0,1]\setminus N$
such that $t_k\nearrow t$, we obtain that $E_t = \bigcap_k E_{t_k}$ and hence $|\chi_{E_{t_k}}-\chi_{E_t}|\to 0$ in $L^1(X)$ as $k\to\infty$.
The lower semicontinuity of the $\BV$ energy with respect to the $L^1$-convergence yields that
\begin{align*}
  P(E_{t}, \overline{\Om}) + P(F, X \setminus \overline{\Om}) & = P(E_{t}, X) \\
  & \le \liminf_{k\to\infty} P(E_{t_k}, X) = \liminf_{k\to\infty} P(E_{t_k}, \overline{\Om}) + P(F, X \setminus \overline{\Om}).
\end{align*}
Hence, $P(E_t, \overline{\Om}) \le \|Du_{\chi_F}\|(\overline{\Om})$. In other words, $t\notin N$.
\end{proof}
\begin{lem}\label{lem:comparison-principle-wk-sol}
Let $F_1 \subset F_2 \subset X$ be sets of finite perimeter in $X$. Suppose that
$E_1, E_2 \subset X$ are chosen such that $\chi_{E_1}$ and $\chi_{E_2}$ are weak solutions to the Dirichlet problem in $\Om$
with boundary data $\chi_{F_1}$ and $\chi_{F_2}$, respectively. Then, $\chi_{E_1 \cap E_2}$ is a weak 
solution corresponding to $\chi_{F_1}$, while $\chi_{E_1 \cup E_2}$ is a weak solution corresponding to $\chi_{F_2}$.
\end{lem}
\begin{proof}
From~\cite[Proposition~4.7(3)]{Mr}, together with the fact that the perimeter
measure is a Borel regular outer measure, we know that
\begin{equation}
  \label{eq:P-cap-cup}
P(E_1\cap E_2,\overline{\Om})+P(E_1\cup E_2,\overline{\Om})
 \le P(E_1,\overline{\Om})+P(E_2,\overline{\Om}).
\end{equation}
If $P(E_1\cap E_2,\overline{\Om})>P(E_1,\overline{\Om})$, then we would have
$P(E_1\cup E_2,\overline{\Om})<P(E_2,\overline{\Om})$. However, this would violate the minimality of 
$P(E_2,\overline{\Om})$ among all BV functions that equal $\chi_{F_2}$ outside $\Om$ since 
$(E_1\cup E_2)\setminus\Om = (F_1\cup F_2)\setminus\Om = F_2\setminus\Om$.
Hence, $P(E_1\cap E_2,\overline{\Om})\le P(E_1,\overline{\Om})$. Furthermore, 
${(E_1\cap E_2)\setminus\Om} = {(F_1\cap F_2)\setminus\Om} = {F_1\setminus\Om}$ and hence 
$\chi_{E_1 \cap E_2}$ is a weak solution to the Dirichlet problem with boundary data $\chi_{F_1}$.

By a similar argument, we can rule out the inequality $P(E_1\cup E_2,\overline{\Om})>P(E_2,\overline{\Om})$ 
as it would violate the fact that $\chi_{E_1}$ is a weak solution for the boundary data $\chi_{F_1}$. 
Therefore, $P(E_1\cup E_2,\overline{\Om})\le P(E_2,\overline{\Om})$ and we conclude that
$\chi_{E_1 \cup E_2}$ is a weak solution to the Dirichlet problem with boundary data $\chi_{F_2}$.
\end{proof}
\begin{remark}
If $F_1 \subset F_2$ are as in Lemma~\ref{lem:comparison-principle-wk-sol} and if $u_{\chi_{F_1}}$ and 
$u_{\chi_{F_2}}$ are weak solutions to the Dirichlet problem with boundary data $\chi_{F_1}$ and $\chi_{F_2}$, 
respectively, then one can use the coarea formula to prove that $\min\{u_{\chi_{F_1}}, u_{\chi_{F_2}}\}$ and 
$\max\{u_{\chi_{F_1}}, u_{\chi_{F_2}}\}$ are weak solutions corresponding to boundary data $\chi_{F_1}$ 
and $\chi_{F_2}$, respectively.
\end{remark}
\begin{defn}
A (weak) solution $\chi_{E}$ to the Dirichlet problem with boundary data $\chi_{F}$ is called a 
\emph{minimal (weak) solution} to the said problem if
every (weak) solution $\chi_{\widetilde{E}}$ corresponding to the data $\chi_{F}$ satisfies $E \ssub \widetilde{E}$, that is, $\mu(E\setminus \widetilde{E})=0$, or alternatively,
$\chi_E \le \chi_{\widetilde{E}}$ $\mu$-a.e.\@ in $X$.
\begin{figure}[ht]%
\center
\includegraphics[width=96mm,height=53mm,page=1]{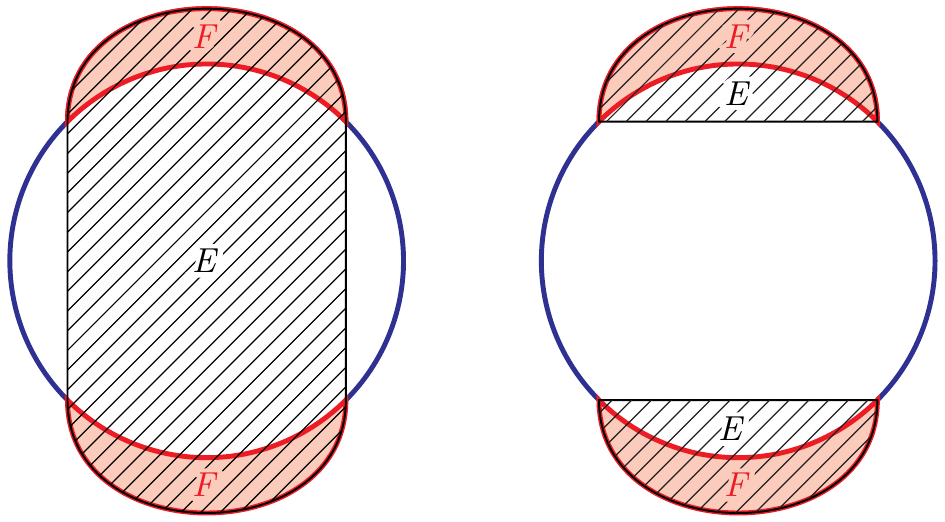}%
\caption{Two weak solutions $\chi_E$ to the Dirichlet problem in $\Om$ with boundary data $\chi_F$, where $F$ is 
the (closed) region filled with light red color. The figure on the right shows the minimal weak solution. Each of the arcs of 
$\dOm \cap \partial F$ and $\dOm \setminus \partial F$ covers the angle of $\pi/2$. Note also that the restriction 
$\chi_{E}\bigr|_\Om$ is a solution\,/\,the minimal solution.}%
\label{fig:min-wk-solution}%
\end{figure}
\end{defn}
\begin{remark}
It follows from Lemma~\ref{lem:soln-vs-minimalSet} that if $u_{\chi_F}$ is a weak solution and $\chi_E$ is the minimal 
weak solution to the Dirichlet problem with boundary data $\chi_F$, then $u_{\chi_F} \ge 1$ a.e.\@ on $E$.
\end{remark}
\begin{prop}
\label{pro:minSolutionSet-exists}
Let $F\subset X$ be a set of finite perimeter in $X$. Then, there exists a unique minimal weak solution $\chi_E$ to the Dirichlet problem 
in $\Om$ with boundary data $\chi_{F}$.
\end{prop}
Here, by uniqueness we mean that two minimal weak solutions agree $\mu$-almost everywhere in $X$.
\begin{proof}
Let $\alpha = \inf_E \mu(E\cap \Om)$, where the infimum is taken over all sets $E$ such that
$\chi_E$ is a weak solution to the Dirichlet problem. Note that there is at least one such weak solution by 
Lemma~\ref{lem:soln-vs-minimalSet}.
Moreover,
$\alpha < \infty$ since $\mu(\Om)<\infty$.

Let $\{E_k\}_{k=1}^\infty$ be a sequence of sets such that $\chi_{E_k}$ solves the Dirichlet problem and 
$\mu(E_k\cap\Om) \to \alpha$ as $k\to \infty$.
Let $\widetilde{E}_1 = E_1$ and $\widetilde{E}_{k+1} = E_{k+1} \cap \widetilde{E}_k$, $k=1,2,\ldots$.
By Lemma~\ref{lem:comparison-principle-wk-sol}, each of the sets $\widetilde{E}_{k}$ gives a weak solution with 
the same boundary data $\chi_F$. Moreover, $\widetilde{E}_{k+1} \subset \widetilde{E}_{k}$ for all $k=1,2,\ldots$ and 
$\mu(\widetilde{E}_k\cap\Om) \to \alpha$.

Let $E = \bigcap_k \widetilde{E}_k$. Then, 
$E\setminus \Om = \bigcap_k (\widetilde{E}_k \setminus \Om) = F \setminus \Om$. 
As $X \setminus \overline{\Om}$ is open and $\chi_E = \chi_F = \chi_{\widetilde{E}_k}$ 
in $X \setminus \overline{\Om}$, we have
\[
  P(E, X \setminus \overline{\Om}) = P(F, X \setminus \overline{\Om}) 
  = P(\widetilde{E}_k, X \setminus \overline{\Om}) \quad\text{for every }k=1,2,\ldots.
\]
Since $|\chi_{\widetilde{E}_k} - \chi_E|\to 0$ in 
$L^1(X)$, the lower semicontinuity of the BV energy \eqref{eq:BV-lsc} yields that
$P(E,X)<\infty$ and then also
\begin{align*}
  P(E, \overline{\Om}) & = P(E, X) - P(F, X \setminus \overline{\Om})
   \le \liminf_{k\to\infty} P(\widetilde{E}_k, X) - P(F, X \setminus \overline{\Om}) \\
   & = \liminf_{k\to\infty} P(\widetilde{E}_k, \overline{\Om})  = \inf\{\|Du\|(\overline{\Om}) \colon u  = \chi_F\text{ in } X \setminus \Om\}.
\end{align*}
Thus, $\chi_E$ is a weak solution to the Dirichlet problem.  If $E'$ is another weak solution, then,
by Lemma~\ref{lem:comparison-principle-wk-sol},
so is $E\cap E'$, and hence $\alpha\le \mu(E\cap E'\cap\Om)\le \mu(E\cap\Om)=\alpha$. Therefore, $\mu(E\setminus E')=0$, that is, $E$ is a minimal weak solution.
The uniqueness now follows
from the above argument, which yields that $\mu(E \symdiff E')=0$ whenever $E'$
is another minimal weak solution.
\end{proof}
\begin{lem}
\label{lem:nesting-minsol}
Let $F_1 \ssub F_2\subset X$ be  
sets of finite perimeter in $X$. Then, the minimal weak solutions $\chi_{E_1}$ and $\chi_{E_2}$ to 
the Dirichlet problem in $\Om$ with boundary data $\chi_{F_1}$ and $\chi_{F_2}$, respectively, satisfy 
$E_1\ssub E_2$.
\end{lem}
\begin{proof}
By replacing $F_2$ with $F_2\cup F_1$ if necessary (and in doing so, we only modify $F_2$ on a set of measure zero),
we may assume that $F_1\subset F_2$.
Let $E_1$ and $E_2$ be as in the statement of the lemma. By 
Lemma~\ref{lem:comparison-principle-wk-sol}, $E_1 \cap E_2$ gives a weak solution 
to the Dirichlet problem with boundary data $\chi_{F_1}$.
Uniqueness of the minimal weak solutions implies that $\mu(E_1 \setminus E_2) = \mu(E_1 \setminus (E_1 \cap E_2)) = 0$.
\end{proof}
We will see in Proposition~\ref{pro:weak-to-strong-positiveCurve} that for domains 
with boundary of positive mean curvature, there is no need to distinguish between solutions and weak solutions for 
boundary data $\chi_F$. Hence, in such domains, there exists a unique minimal solution, and furthermore,
the minimal solutions exhibit the same nesting property for
nested boundary data as in Lemma~\ref{lem:nesting-minsol}.

It is, in fact, also possible to define a \emph{maximal (weak) solution} $\chi_E$ to the Dirichlet problem in $\Om$ 
with boundary data $\chi_F$ by requiring that $\mu(\widetilde{E}\setminus E) = 0$ for every other (weak) solution 
$\chi_{\widetilde E}$ of the said problem. For instance, the set $E$ on the left in Figure~\ref{fig:min-wk-solution} 
gives the maximal (weak) solution.
\section{Domains in metric spaces with boundary of positive mean curvature}
In this section we propose a notion of positive mean curvature of the boundary of a domain
$\Om$ in the metric measure space $X$, and study the Dirichlet problem for such domains.
As explained in the introduction, solutions to Dirichlet problem in the sense of
Definition~\ref{def:strong-solutions} might not
always exist.
Given an open set $F\subset X$ that intersects $\partial\Omega$, let $u_{\chi_F}$ denote a generic
weak solution to the 
Dirichlet problem with boundary data $\chi_F$. It is not necessarily true that
$\Tr u_{\chi_F}=\chi_F$ $\Hcal$-a.e.~on $\dOm$. 
The classic example is that of a square.
If $\Omega=[0,1]\times[0,1]\subset\R^2$, and if $F$ is the disk centered at $(1/2,0)$ with radius
$1/10$, then there is no function $u$ of least gradient in $\Omega$ with trace $\chi_F$ on
$\partial\Omega$. Notice that the boundary of the square does not have positive mean curvature.

In the definition of positive mean curvature below (Definition~\ref{posmeancur}), we tacitly require that
for each $z\in\partial\Omega$ and \emph{almost all} $0<r<r_0$, the
function $u_{\chi_{B(z,r)}}$ exists. This is not an onerous assumption, as seen from Lemma~\ref{lem:existence-weak-sol}
and the fact that given $x\in X$, the ball $B(x,r)$ has finite perimeter in $X$ for almost every $r>0$.
This latter fact follows from the coarea formula~\eqref{eq:coarea-formula}.

The main question addressed in this paper is the following.
\begin{quest}
\emph{If $\Omega$ has a boundary $\partial\Omega$ with positive mean curvature
(in the sense of Definition~\ref{posmeancur} below), is it
true that for every Lipschitz function $f:\partial\Omega\to\R$, there exists an extension of least gradient
$u:\Omega\to\R$ such that $f$ is the trace of $u$, that is,
\[
\lim_{r\to 0}\fint_{B(z,r)\cap\Omega}|u-f(z)|d\mu=0
\]
for $\Hcal$-almost every $z\in \partial\Omega$?
In other words, does there exist a solution to the Dirichlet
problem in the sense of Definition~\ref{def:strong-solutions} with boundary data $f$?
If such solutions exist, can we guarantee that they will be continuous and unique?}
\end{quest}
We will show that indeed solutions do exist, and by counterexamples we give a negative
answer to the continuity and uniqueness questions.

The hypothesis of positive mean curvature of the boundary seems appropriate in view of the 
results of~\cite{SWZ} in the Euclidean setting, where existence, continuity and uniqueness 
of solutions was proven for bounded Lipschitz domains $\Omega\subset \R^n$ provided that:
\begin{enumerate}
  \renewcommand{\labelenumi}{(\arabic{enumi})}
  \setlength{\parskip}{0.5pt plus 0.5pt minus 0.5pt}
  \setlength{\itemsep}{0.5pt plus 1pt minus 0.5pt}
	\item $\partial\Omega$ has non-negative mean curvature (in a weak sense),
  \item $\partial\Omega$ is not locally area-minimizing.
\end{enumerate}
Moreover, if $\partial\Omega$ is smooth, then these two conditions together are equivalent
to the fact that the mean curvature of $\partial\Omega$ is positive on a dense subset of $\partial\Omega$.

To talk about traces of solutions as referred to above, we need to know that such traces exist.
It is not difficult to construct Euclidean domains and BV functions on the Euclidean domains that fail to
have a trace on the boundary of the domain.
In the metric setting (which also includes the Euclidean setting), 
it was shown in~\cite{LS} that there exist traces of BV functions,
as defined in~\eqref{eq:trace-def}, on the boundary of
bounded domains satisfying the conditions listed on page 
\pageref{enum:conditions for trace} of the present paper.
In this paper, we \emph{do not} need
to know that \emph{every} BV function on the domain of interest has a trace on the boundary. 
We are only interested in knowing whether the weak solutions we construct have the
correct trace.
\begin{defn}\label{posmeancur}
Given a domain $\Omega\subset X$, we say that the boundary $\partial\Omega$
has {\em positive mean curvature} if there exists a
non-decreasing function $\varphi:(0,\infty)\to(0,\infty)$ and a constant $r_0>0$ such that for
all $z\in\partial\Omega$ and all $0<r<r_0$ with $P(B(z,r),X)<\infty$
we have that
$u_{\chi_{B(z,r)}}^\vee \geq 1$ everywhere on $B(z,\varphi(r))$.
Since the weak solution $u_{\chi_{B(z,r)}}$ need not
be unique, the above condition is required to hold
for \emph{all} such solutions.
\end{defn}
Note that the requirement on all weak solutions $u_{\chi_{B(z,r))}}$ in the definition above can be
equivalently expressed as the condition that $B(z, \varphi(r)) \ssub E_{B(z,r)}$, where $E_{B(z,r)}\subset X$
gives the~\emph{minimal weak solution} to the Dirichlet problem with boundary data
$\chi_{B(z, r)}$ as given by Proposition~\ref{pro:minSolutionSet-exists}.
\begin{remark}\label{rem:curv-diff-defn}
Our definition of $\partial\Om$ being of positive mean curvature is different from that of~\cite{SWZ}. In~\cite{SWZ},
it is required that
\begin{itemize}
\item for each $x\in\dOm$ there is some $\eps_0>0$ such that whenever $A\Subset B(x,\eps_0)$ with $P(A,\R^n)<\infty$,
we must have $P(\Om,\R^n)\le P(\Om\cup A,\R^n)$, and
\item for each $x\in\dOm$ there is some $\eps_1>0$ such that whenever $0<\eps<\eps_1$, there is some 
$A_\eps\Subset B(x,\eps)$ such that $P(A_\eps,\R^n)$ is finite and $P(\Om\setminus A_\eps,\R^n)< P(\Om,\R^n)$.
\end{itemize}
In the case of $\dOm$ being a smooth manifold, the two definitions coincide.
\end{remark}
\noindent%
\begin{minipage}{\linewidth}
\vspace{\topsep}
\center
\includegraphics[width=4.75in,page=2]{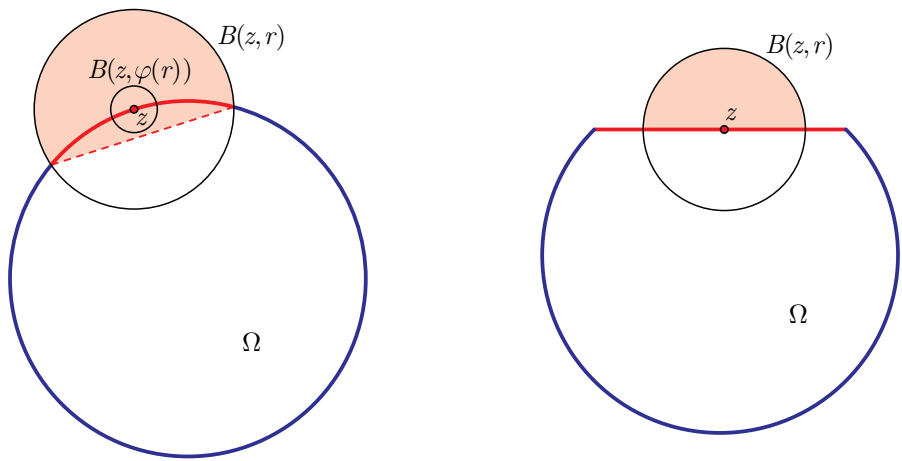}

\captionof{figure}{The (Euclidean) domain on the left has boundary of positive mean curvature, unlike the
domain on the right. The regions shaded light red in color are weak solution sets of the respective
Dirichlet problems.}\label{fig:circles}
\vspace{\topsep}
\end{minipage}

Euclidean balls of radii $R>0$ satisfy the above condition, with
$\varphi(r)=\tfrac{r^2}{2R}$, as can be seen via a
simple computation.
On the other hand, the square region
$\Om=(0,1)\times(0,1)\subset\R^2$ does not satisfy the criterion of positive mean curvature of the boundary. Indeed, for
$z=(1/2,0)$ and $0<r<1/2$, the weak solution
is $u_{\chi_{B(z,r)}} = \chi_{B(z,r)\setminus\Om}$. For the same reason the domain
obtained by removing a slice from the disk also does not satisfy the criterion for
positive mean curvature of the boundary, see Figure~\ref{fig:circles}.
\begin{exa}
Consider $\Om = B(0,1) \subset \R^2$ with the weighted measure $d\mu=w\, d\Leb2$. Define the following distance
\[
\hat{d}(x,y)=\inf_{\gamma}\int_{\gamma} w^{1/2}\,ds,
\]
where the infimum is taken over all the curves $\gamma$ connecting $x$ and $y$.

If $\Omega$ is the disk with the Euclidean metric and weighted measure, then the boundary will have
positive mean curvature in the sense of Riemannian geometry but might not be of positive mean curvature in our sense.

If we consider $(\Omega,\hat{d})$, then $\partial\Omega$ might not have positive mean curvature in the Riemannian geometric sense either. Indeed, it will fail to be of positive mean curvature if the weight function decreases rapidly towards the boundary of the disk.

If $\Om$ is the ``flattened disk'' as in Figure~\ref{fig:circles} and the weight function \emph{increases} rapidly towards the flattened part of the boundary of that domain, then even though this boundary is not
of positive mean curvature in the Riemannian sense, it \emph{would be} of positive mean curvature in our sense. Thus the notion of curvature is intimately connected with the interaction between the metric and the measure.
\end{exa}
\begin{exa}
Assume that $X$ is the unit sphere $\mathbb{S}^2$, equipped with the spherical metric and the $2$-dimensional Hausdorff measure. 
Let $x_0\in X$, and consider $\Om_R=B(x_0,R)$ for $0<R<\pi$. We show that $\Om_R$ has boundary of positive mean curvature (in our sense) precisely when $0<R<\pi/2$.

Let $z\in \dOm_R$ and $0<r<\diam \Om_R$. Then weak solutions $u_{\chi_{B(z,r)}}$ of the Dirichlet problem in $\Om_R$ with boundary data $\chi_{B(z,r)}$ have superlevel sets $E_{t}=\{x:u_{\chi_{B(z,r)}}(x)>t\}$ of minimal boundary surface. For any $0<t<1$, $\partial E_{t}$ consists of the shortest path in $\overline{\Om_R}$ which connects the two points in $\partial B(z,r) \cap \dOm_R$.

Suppose that $R<\pi/2$. Then the shortest path $\gamma$ in $\overline{\Om_R}$ connecting
the two components (points) of
$\partial B(z,r) \cap \dOm_R$ is part of a great circle. It is clear from the geometry that there exists a 
positive function $\varphi(r)$, independent of $z$, such that $B(z,\varphi(r))\cap \Om_{R}\cap \gamma=\emptyset$. Hence $E_{t}\supset B(z,\varphi(r))$ for any $0<t<1$, which implies $u_{\chi_{B(z,r)}}\geq 1$ on $B(z,\varphi(r))$. This shows that 
the boundary of $\Om_R$ has positive mean curvature.

If instead $R \ge \pi/2$, then the shortest path in $\overline{\Om_R}$ connecting
the two components of
$\partial B(z,r) \cap \dOm_R$ lies entirely in $\dOm_R$. Hence $E_{t}\cap \Om_R=\emptyset$ for any $0<t<1$. This implies that the weak solution $u_{\chi_{B(z,r)}}$ is exactly $\chi_{B(z,r) \setminus \Om_R}$. Hence there is no 
positive function $\varphi(r)$ as in
Definition~\ref{posmeancur}, so $\Om_{R}$ is not of positive mean curvature.

Observe that, for $R \ge \pi/2$, the weak solution $u = \chi_{B(z,r)\setminus \Om_R}$ is not a solution. Indeed, $\Tr u \equiv 0 \neq \chi_{B(z,r)}$ on $\dOm_{R}$. In fact, there is no solution for such a boundary data since $\inf\{\|Du\|(\Om_{R}) \colon \Tr u = \chi_{B(z,r)}\text{ on }\dOm_{R}\} = \Hcal( B(z,r) \cap \dOm_{R})$ is not attained by any function $u\in \BV(\Om_{R})$.
\end{exa}
To prove the main result of this paper, Theorem~\ref{thm:main}, we need the following tools.
\begin{lem}\label{lem:Not-Close-to-bdry}
Assume that $\dOm$ has positive mean curvature. Let
$F\subset X$ be a set of finite perimeter in $X$.
Suppose that $x\in\dOm$ and $0<r<r_0$ such that $B(x,r)\setminus \Om \subset F$
with $P(B(x,r),X)<\infty$. Assume that $u_{\chi_F}$ is a weak solution to the Dirichlet problem in $\Om$ 
with boundary data $\chi_F$. Then, $B(x,\varphi(r))\subset \{u_{\chi_F}^\vee\ge 1 \}$, where 
$\varphi: (0, \infty)\to(0,\infty)$ is the function of the condition of positive mean curvature from Definition~\ref{posmeancur}.
\end{lem}
\begin{proof}
By Lemma~\ref{lem:soln-vs-minimalSet}, there is a set $G \subset X$ of finite perimeter such that $\chi_G$ 
is a weak solution to the Dirichlet problem with boundary data $\chi_{B(x,r)}$. Furthermore, 
Lemma~\ref{lem:soln-vs-minimalSet} yields that 
$E_1 \coloneq \{x\in X  \colon u_{\chi_F} \ge 1\}$
is a weak solution set corresponding to boundary data $\chi_F$. 

By Lemma~\ref{lem:comparison-principle-wk-sol}, $\chi_{E_1 \cap G}$ is a weak solution corresponding 
to boundary data $\chi_{B(x,r)}$. 
Then, $B(x, \varphi(r)) \ssub E_1 \cap G$ by the definition of positive mean curvature. In particular, 
$B(x,\varphi(r)) \ssub E_1$. 
Therefore, $u_{\chi_F}^\vee \ge 1$ everywhere on $B(x,\varphi(r))$.
\end{proof}
Combining Lemma~\ref{lem:soln-vs-minimalSet} with the lemma above tells us that there is 
\emph{at least one weak solution set} to the Dirichlet problem with boundary data $\chi_F$ and that 
\emph{every weak solution set} to this boundary data contains the ball $B(x,\varphi(r))$ whenever 
$x\in \dOm$ and $0<r<r_0$ such that
$B(x,r)\setminus \Om\subset F$.
\begin{cor}\label{cor:correct-trace}
Suppose that $\dOm$ is of positive mean curvature, 
and let $F\subset X$ be open with $P(F,X)<\infty$ and $\Hcal(\dOm\cap \partial F)=0$. Suppose that 
$u_{\chi_F}$ is a weak solution to the Dirichlet problem with boundary data $\chi_F$. Then, 
$\Tr u_{\chi_{F}} = \chi_F$ $\Hcal$-a.e.~on $\dOm$.
\end{cor}
\begin{proof}
By the maximum principle~\cite[Theorem 5.1]{HKLS}, we know that $0 \le u_{\chi_F} \le 1$. For every 
$x\in \dOm\cap F$, there is $r_x > 0$ such that $B(x,r_x) \setminus \Om \subset F\setminus \Om$. 
Thus, we can apply Lemma~\ref{lem:Not-Close-to-bdry} to find a ball $B(x, \varphi(r_x))$ such that 
$u_{\chi_F}^\vee \ge 1$ everywhere on $B(x, \varphi(r_x))$. 
Hence, $\Tr u_{\chi_F}(x) = 1$.

Note that $1-u_{\chi_F}$ is a weak solution to the Dirichlet problem with boundary data $\chi_{X\setminus F}$. 
Hence, for every $x \in \dOm \setminus \overline{F}$, Lemma~\ref{lem:Not-Close-to-bdry} provides us with a ball 
$B(x, \varphi(r_x))$ such that $1-u_{\chi_F}^\wedge \ge 1$, i.e., 
$u_{\chi_F}^\wedge \le 0$ everywhere on $B(x, \varphi(r_x))$. 
Hence, $\Tr u_{\chi_F}(x) = 0$.

Finally, even though we lack any control of $\Tr u_{\chi_F}$ on $\dOm\cap \partial F$, the proof is complete 
since we assumed that $\Hcal(\dOm\cap \partial F)=0$.
\end{proof}
\begin{lem}\label{lem:ext-of-trace-global}
Suppose that $\Hcal(\dOm)<\infty$.
Let $F\subset X$
be an open set such that $\Hcal(\dOm\cap\partial F)=0$ and $P(F,X)<\infty$.
If $v\in \BV(\Om)$ with
$0\le v\le 1$ and $\Tr v=\chi_F$ $\Hcal$-a.e.~in $\dOm$, 
then the extension of $v$ to $X\setminus \Om$ obtained by defining $v=\chi_F$ in
$X\setminus\Om$
lies in $\BV_{\loc}(X)$ and $\| Dv\|(\dOm)=0$.
\end{lem} 
\begin{proof}
Let $v$ be extended to $X\setminus\Om$ by setting it to be equal to $\chi_{F}$ there. A priori, we know only that $v\in \BV(\Om)$, and
so we need to show that the extended function, also denoted $v$,
belongs to $\BV(X)$. To this end, we employ
the coarea formula. Recall that  $0\le v\le 1$. For $0<t<1$, let
$E_t = \{x\in X\colon v(x)>t\}$. Then
\[
 E_t=(E_t\cap\Om)\cup (F\setminus\Om).
\]
Observe that $\chi_{F\setminus\Om}=\chi_{F} \chi_{X\setminus\Om}$, and hence
$P(F\setminus \Om, X)<\infty$ by the assumptions that $\Hcal(\dOm)<\infty$
(which implies that $P(\Om,X)<\infty$) and
$P(F,X)<\infty$. Thus, in order to gain
control over $P(E_t,X)$, we only need to control $P(E_t\cap\Om,X)$.
For every $\eps>0$,
we can cover the compact set $\dOm$ by finitely many balls $B_i$, $i=1,\cdots, k$, with radii $r_i<\eps$ such that
\begin{enumerate}
  \renewcommand{\theenumi}{(\arabic{enumi})}
  \renewcommand{\labelenumi}{\theenumi}
  \setlength{\parskip}{0.5pt plus 0.5pt minus 0.5pt}
  \setlength{\itemsep}{0.5pt plus 1pt minus 0.5pt}
  \item $P(B_i,X)\le 2C_D \tfrac{\mu(B_i)}{r_i}$ for each $i$,
  \label{it:PBi-vs-muBi/ri}
  \item $\sum_i \tfrac{\mu(B_i)}{r_i}\le C_D(1+\eps)\Hcal(\partial\Om)$.
\end{enumerate}
We now show how to find such a cover.

An application of the coarea formula applied to the function $u(x) = d(z, x)$ for some fixed $z \in X$ gives that if $r>0$, then
\begin{align}
\notag
  \mu(B(z,2r)) \ge \|Du\|(B(z,2r)) & = \int_0^{2r} P(\{u>t\}, B(z,2r))\,dt \\
  & \ge \int_{r}^{2r} P(B(z,t), X)\,dt \ge r \, P(B(z,r_0), X)
\label{eq:coarea-muB/r-vs-PB}
\end{align}
for some $r_0 \in [r, 2r]$. In order to find balls $B_i$ of the desired properties,
we cover $\dOm$ by finitely many balls $B(z_i, R_i)$ with radius $R_i < \eps/2$
so that
\[
\sum_i \frac{\mu(B(z_i, R_i))}{R_i} < (1+\eps) \Hcal(\dOm).
\]
By~\eqref{eq:coarea-muB/r-vs-PB}, there is $r_i \in [R_i, 2R_i]$ such that
\[
P(B(z_i,r_i),X) \le \frac{\mu(B(z_i, 2R_i))}{R_i} \le \frac{2C_D \mu(B(z_i, r_i))}{r_i}.
\]
Setting $B_i = B(z_i, r_i)$ yields that $\sum_i \frac{\mu(B_i)}{r_i} \le \sum_i \frac{\mu(B(z_i, 2R_i))}{R_i} \le  C_D(1+\eps) \Hcal(\dOm)$.

We now set $U_{\eps,t}=(E_t\cap\Om)\setminus\bigcup_{i=1}^k B_i$.
Note that as $r_i<\eps$,
\[
\mu\biggl(\bigcup_iB_i\biggr)\le \sum_i\mu(B_i)\le \eps \sum_i \frac{\mu(B_i)}{r_i}
  \le \eps C_D (1+\eps)\Hcal(\partial\Om)\to 0\quad\text{as }\eps\to 0^+.
\]
Therefore $\chi_{U_{\eps,t}}\to\chi_{E_t}$ in $L^1(X)$ as $\eps\to0^+$. Since $U_{\eps, t}$ is compactly contained in $\Om$, we can estimate
\begin{align*}
P(U_{\eps,t},X)=P(U_{\eps,t},\Om)&\le P(E_t,\Om)+\sum_{i=1}^kP(B_i,X)\\
  &\le P(E_t,\Om)+C\, \sum_{i=1}^k\frac{\mu(B_i)}{r_i}
	\le P(E_t,\Om)+C(1+\eps)\Hcal(\dOm).
\end{align*}
The lower semicontinuity of the BV energy with respect to
$L^1$-convergence gives that
\[
P(E_t\cap\Om,X)\le \liminf_{\eps \to 0^+} P(U_{\eps,t},X) \le P(E_t,\Om)+C\, \Hcal(\dOm).
\]
Thus by the coarea formula,
\begin{align*}
\| Dv\|(X)=\int_0^1P(E_t,X)\, dt
   &\le \int_0^1[P(E_t\cap\Om,X)+P(F\setminus\Om,X)]\, dt\\
	 &\le \int_0^1[P(E_t,\Om)+C\Hcal(\dOm)+P(F\setminus\Om,X)]\, dt\\
	 &\le \| Dv\|(\Om)+C\Hcal(\dOm)+P(F\setminus\Om,X)<\infty.
\end{align*}
Hence $v\in \BV(X)$.

Finally, since $\Tr v=\chi_F$ $\Hcal$-a.e.\@ on $\dOm$,
the jump set $S_v$ of $v$ satisfies $\Hcal(\dOm\cap S_v\setminus\partial F)=0$
and hence
$\Hcal(S_v\cap\dOm)\le \Hcal(\partial F\cap\dOm)=0$.
Therefore, $\| Dv\|(\dOm)=0$,
recall the decomposition \eqref{eq:variation measure decomposition}
and the discussion
after it.
\end{proof}
Now we compare weak solutions and solutions for bounded domains whose boundary has
positive mean curvature.
\begin{prop}\label{pro:weak-to-strong-positiveCurve}
Suppose that $\dOm$ is of positive mean curvature and that $\Hcal(\dOm)$ is finite. Let
$F\subset X$ be open with $P(F,X)<\infty$ and $\Hcal(\partial F\cap\dOm)=0$. Then,
all weak solutions $u_{\chi_{F}}$
are also solutions, so that if $v\in \BV(\Om)$ with $\Tr v=\chi_{F}$ $\Hcal$-a.e.~on $\partial\Om$, then
\[
  \| Du_{\chi_F}\|(\Om)\le \| Dv\|(\Om).
\]
\end{prop}
Note that if $f$ is a continuous BV function on $X$, then for almost every $t\in\R$, the set $F=\{x\in X\, :\, f(x)>t\}$ satisfies the hypotheses of the above proposition. This follows from the coarea formula and the fact that $\Hcal(\dOm)<\infty$.
\begin{proof}
For the sake of ease of notation, set $u= u_{\chi_{F}}$.
Note that as $\dOm$ is of positive mean curvature,
$\Tr u_{\chi_{F}} =\chi_{F}$ $\Hcal$-a.e.~in
$\dOm$ by Corollary~\ref{cor:correct-trace}.
Moreover, by the maximum principle~\cite[Theorem 5.1]{HKLS},
we know that $0\le u\le 1$.
Hence by Lemma~\ref{lem:ext-of-trace-global}
we have $u\in \BV_{\loc}(X)$ (which comes for free as $u$ is a weak solution) 
with $\| Du\|(\dOm)=0$.

If $v\in \BV(\Om)$ with $\Tr v=\chi_{F}$ $\Hcal$-a.e.~in
$\dOm$, then we can assume that $0 \le v \le 1$, since
truncations do not increase BV energy and the truncation
$\min\{1, \max\{0, v\}\}$ also has the correct trace $\chi_{F}$ on $\dOm$.
By Lemma~\ref{lem:ext-of-trace-global} again we know that the extension of $v$
to $X\setminus\Om$ by $\chi_F$ gives a function in $\BV_{\loc}(X)$ with
$\| Dv\|(\dOm)=0$. Now,
\[
\| Du\|(\overline{\Om})\le \| Dv\|(\overline{\Om})
\]
by the fact that $u$ is a weak solution to the Dirichlet problem on $\Om$ with boundary data $\chi_{F}$. Then,
\begin{align*}
\| Du\|(\Om) = \| Du\|(\overline\Om) -\| Du\|(\dOm) & \le \| Dv\|(\overline\Om) -\| Du\|(\dOm)\\
& = \| Dv\|(\overline\Om) -\| Dv\|(\dOm) = \| Dv\|(\Om).
\qedhere
\end{align*}
\end{proof}
The previous proposition tells us that in using weak solutions we do obtain (strong) solutions.
The next proposition completes the picture regarding the relationship between
the notions of solutions and weak solutions to the Dirichlet problem, by showing that
the \emph{only way} to obtain (strong) solutions is through weak solutions.
\begin{prop}
\label{pro:strong-is-weak}
Suppose 
that $\Hcal(\dOm)<\infty$. 
Let $F\subset X$
be an open set such that $\Hcal(\dOm\cap\partial F)=0$ and $P(F,X)<\infty$.
If $v\in\BV(\Om)$ is a solution to the Dirichlet problem with boundary data $\chi_F$, then
the extension of $v$ by $\chi_F$ outside $\Om$ is a weak solution to the said Dirichlet problem.
\end{prop}
\begin{proof}
Let $v\in \BV(\Om)$ be a solution to the Dirichlet problem with boundary
data $\chi_F$. Then $\Tr v=\chi_F$ $\Hcal$-a.e.~in $\dOm$, and so by
Lemma~\ref{lem:ext-of-trace-global}, the extension of $v$ by $\chi_F$ to
$X\setminus\Om$, also denoted $v$, lies in $\BV_{\loc}(X)$ with $\| Dv\|(\dOm)=0$. In particular,
$\| Dv\|(\overline{\Om})=\| Dv\|(\Om)$.

Let $E\subset X$ be a weak solution set for the boundary data $\chi_F$.
The existence of such a set is guaranteed by Lemma~\ref{lem:soln-vs-minimalSet}.
Then, $\Tr \chi_E = \chi_F$ $\Hcal$-a.e.\@ in $\dOm$ by Corollary~\ref{cor:correct-trace}. Since $v$ is a solution to the Dirichlet problem on $\Om$ with boundary data
$\chi_F$, it follows that
\[
\| Dv\|(\overline{\Om})=\| Dv\|(\Om)
\le \| D\chi_E\|(\Om)\le \| D\chi_E\|(\overline{\Om}).
\]
The fact that $\chi_E$ is a weak solution to the Dirichlet problem on
$\Om$ with boundary data $\chi_F$ tells us that $v$ is also a weak solution, since
$\| Dv\|(\overline{\Om})\le \| D\chi_E\|(\overline{\Om})
\le \| Dw\|(\overline{\Om})$
whenever $w\in \BV_{\loc}(X)$ with $w=\chi_F$ on $X\setminus\Om$.
\end{proof}
If $\Om\subset X$ is a bounded domain such that $\Hcal(\dOm)<\infty$ and with $\dOm$ of positive mean curvature, then Propositions~\ref{pro:weak-to-strong-positiveCurve} and~\ref{pro:strong-is-weak} together tell us that weak solutions to the Dirichlet problem with boundary data $\chi_F$ are solutions to the said Dirichlet problem and vice versa, provided that $F\subset X$ is an open set of finite perimeter in $X$ such that $\Hcal(\dOm\cap\partial F)=0$.
Hence, there is no need to distinguish between weak solutions and solutions for such type of Dirichlet boundary data.

Now we are ready to prove the main theorem of this paper, the existence of solutions for continuous boundary data.
While~\cite{SWZ} focuses on Lipschitz boundary data, we consider the larger class,
$\BV_{\loc}(X)\cap \Ccal(X)$, of boundary data. 
The reason why~\cite{SWZ} focused on Lipschitz data was because for such data,
in the Euclidean setting, it was also possible to show that there is a globally Lipschitz solution as well. We will show in the
final section of this paper that even in the most innocuous setting of  weighted Euclidean spaces, such Lipschitz continuity
fails; therefore, there is no reason for us to restrict ourselves to the study of Lipschitz boundary data.
\begin{thm}\label{thm:main}
Suppose that $\Hcal(\dOm)<\infty$
and that $\dOm$ has positive mean curvature. Let
$f \in \BV_{\loc}(X)\cap \Ccal(X)$. Then, there is a solution 
$u\in \BV_{\loc}(X)$ to the Dirichlet problem in $\Om$.
Furthermore,
\[
 \lim_{\Om\ni y\to x}u(y)=f(x)
\]
whenever $x\in \dOm$. 
Moreover, $u$ is a weak solution to the given Dirichlet problem.
\end{thm}
\begin{proof}
Recall from our standing assumptions, listed at the beginning of Section~\ref{sec:prelresults},
that $\Omega$ is bounded. Hence, we can find a ball $B\subset X$ such that $\overline{\Om}\subset B$,
and we can find a Lipschitz function $\varphi:X\to[0,1]$ such that $\varphi=1$ on $B$ and $\varphi=0$ on $X\setminus 2B$.
Replacing $f$ with $f\varphi$ in the above theorem will not change the class of solutions inside $\Om$.
Therefore, we will assume without loss of generality that $f$ is compactly supported and hence bounded, and $f\in \BV(X)\cap \Ccal(X)$.

For $t\in\R$, define $F_t=\{x\in X\colon f(x)>t\}$.
Then, $F_t$ is open by continuity of $f$. Moreover, $F_t = \emptyset$ for sufficiently large $t$,
while $F_t = X$ for sufficiently small $t$.

As $f\in \BV(X)$, the coarea formula \eqref{eq:coarea-formula} yields that $P(F_t,X)<\infty$ for
a.e.\@ $t\in\R$. Since $\partial F_t \cap \partial F_s = \emptyset$ whenever $s\neq t$, the finiteness
of $\Hcal(\dOm)$ implies that $\Hcal(\dOm \cap \partial F_t) = 0$ for all but (at most) countably many $t\in\R$.
Let
\[
  J = \{t \in \R \colon P(F_t,X)<\infty\ \text{and}\ \Hcal(\dOm \cap \partial F_t) = 0\}.
\]
For every $t\in J$, we can apply Proposition~\ref{pro:minSolutionSet-exists} to find a set 
$\widetilde{E}_t \subset X$ that is 
the minimal weak solution set to the Dirichlet problem on $\Om$ with boundary data $\chi_{F_t}$. 
We set $E_t=\{x\in X\, :\, \chi_{\widetilde{E}_t}^\vee(x)>0\}$.
Then, $\chi_{E_t}$ is also a minimal solution. 

By Lemma~\ref{lem:nesting-minsol}, the family of sets $\{E_t\colon t\in J\}$ is nested in the
sense that $E_s \subset E_t$ whenever $s,t \in J$ with
$s>t$, since $F_s \subset F_t$.
As $\Leb1(\R \setminus J) = 0$, we can pick a
countable set $I \subset J$ so that $I$ is dense in $\R$.
Now, we can define $u: X \to\R$ by
\[
u(x)=\sup\{s\in I\colon x\in E_s\}, \quad x\in X,
\]
and show that it satisfies the conclusion of the theorem. 
Observe that $u$ is measurable because
\[
\{x\in X\colon u(x)\ge t\}=\bigcap_{I \ni \sigma < t} E_\sigma, \quad t\in \R,
\]
i.e., all superlevel sets can be expressed as countable intersections of measurable sets.

For $t\in J$, i.e., for a.e.\@ $t\in\R$, we have 
\begin{equation}
  \label{eq:Kt-def}
  K_t \coloneq \{x \in X \colon u(x) > t\}
   =\bigcup_{I\ni s > t} E_s \subset E_t \subset \{x \in X \colon u(x) \ge t\}
    = \bigcap_{I\ni \sigma < t} E_\sigma.
\end{equation}
Since $\mu$ is $\sigma$-finite on $X$, we have 
$\mu\bigl(\{x\in X\colon u(x) = t\}\bigr) = 0$ for all but (at most) countably many
$t\in \R$. In particular, $\mu(K_t \symdiff E_t)=0$ for a.e.\@ $t\in J$, whence for such $t$, the set
$K_t$ is a weak solution set for the
Dirichlet problem with boundary data $\chi_{F_t}$. Considering the fact that $\chi_{F_t}$ is one of the
competitors in the definition of a solution to a Dirichlet problem with boundary data $\chi_{F_t}$, the coarea formula yields that
\begin{align*}
  \|Du\|(X) = \int_\R P(\{u>t\}, X)\,dt & = \int_\R P(K_t, X)\,dt  \\
  & \le \int_\R P(F_t, X)\,dt = \|Df\|(X) < \infty.
\end{align*}
Since $u=f$ in $X \setminus \Om$, it follows that $u\in\BV(X)$. Note that up to this point, we did not need the positive mean curvature property of $\dOm$.

Next, we show that $\lim_{\Om\ni y\to z}u(y)=f(z)$ for $z\in\dOm$ and that
$\Tr u=f$ on $\dOm$. To this end, note that if $z\in \dOm$ and $t\in I$ with $t>f(z)$, then there is some $r_{z,t}>0$ such that $B(z,r_{z,t}) \subset X \setminus \overline{F_t}$.
Then, $B(z, \varphi(r_{z,t})) \cap \Om \subset \Om \setminus E_t$ similarly as in the proof of
Corollary~\ref{cor:correct-trace}. Thus, $u \le t$ on $B(z,\varphi(r_{z,t}))\cap\Om$ for each $I \ni t>f(z)$. Hence,
\[
 \limsup_{\Om\ni y\to z}u(y)\le f(z).
\]
Also, for every $t\in I$, $t<f(z)$, there is $\rho_{z,t}>0$ such that $B(z,\rho_{z,t})\subset F_t$. Hence, 
$B(z,\varphi(\rho_{z,t}))\cap\Om\subset E_t$ for such $t$. Consequently,
$u\ge t$ on $B(z,\varphi(\rho_{z,t}))\cap\Om$. Thus,
\[
 \liminf_{\Om\ni y\to z}u(y)\ge f(z).
\]
Considering that $\lim_{\Om\ni y\to z}u(y) = f(z)$, we can conclude that $\Tr u(z) = f(z)$ directly
from the definition of the trace, see~\eqref{eq:trace-def}.

Next, we show that $u$ is a solution to the Dirichlet problem in $\Om$ with boundary data $f$.
We have already proven that $\Tr u=f$ on $\dOm$. Let $v\in \BV(\Om)$ such that $\Tr v=f$
$\Hcal$-a.e.~on $\dOm$. Then, $\Tr \chi_{\{v>t\}} =\chi_{F_t}$ $\Hcal$-a.e.\@ on $\partial\Om$
for almost every $t\in\R$.
Since $K_t$ is a solution to the Dirichlet problem
with boundary data $\chi_{F_t}$ for a.e.\@ $t\in J$, for such $t$ we can estimate $P(K_t,\Om)\le P(\{v>t\},\Om)$.
By the coarea formula, we obtain that
\[
\| Du\|(\Om)=\int_{\R}P(K_t,\Om)\, dt\le \int_{\R}P(\{v>t\},\Om)\, dt=\| Dv\|(\Om).
\]
Thus, $u$ is a solution to the Dirichlet problem in $\Om$ with boundary data $f$.

Finally, we show that $u$ is a weak solution. This part also does not need the positive mean curvature 
assumption of $\dOm$.
Note that by construction of $u$ and by the continuity of $f$, we have $u=f$ on $X\setminus\Om$.
Assume that $w \in \BV_\loc(X)$ satisfies $w=f$ in $X \setminus \Om$. In order to prove that $u$ is a weak solution to the 
Dirichlet problem, we need to verify that $\|Du\|(\overline\Om) \le \|Dw\|(\overline\Om)$. Recall that for almost every 
$t\in\R$, the set $K_t$ gives a weak solution set for the Dirichlet problem with boundary data $\chi_{F_t}$,
see the discussion following~\eqref{eq:Kt-def}.
In particular, $P(K_t, \overline\Om) \le P(\{w>t\}, \overline\Om)$. The coarea formula then yields that
\[
  \|Du\|(\overline\Om) = \int_{\R} P(K_t, \overline\Om)\,dt \le \int_{\R} P(\{w>t\}, \overline\Om)\,dt = \|Dw\|(\overline\Om)\,,
\]
which concludes the proof that $u$ is a weak solution to the Dirichlet problem with boundary data $f$.
\end{proof}
It might seem at a first casual glance at the  proof above that it suffices to assume
that the boundary data $f$ is semicontinuous.
 The reader should note that this is not
the case; our proof does not work for non-continuous but semicontinuous $f$,
for we need openness of \emph{both} $\{f>t\}$ and $\{f<t\}$ for all $t\in\R$.
This will fail for
non-continuous semicontinuous functions. It might be that the theorem holds also
for semicontinuous functions, but our method of proof will not work for them.
The paper of~\cite{ST} also shows that even in the simple setting of the
Euclidean plane, there are functions $f\in \BV(\R^2)$ for which the
(strong) solution to the Dirichlet problem for least gradient in the Euclidean disk
with boundary data $f$ does not exist. Thus, it is reasonable to restrict our attention
to continuous boundary data.
\begin{remark}
A study of the proof of Theorem~\ref{thm:main} gives the following generalization of this theorem to a wider class of domains.
Given a bounded domain $\Om\subset X$ with $\mu(X\setminus\Om)>0$ and a point $z\in\dOm$, we say that 
$\dOm$ \emph{has positive mean curvature at} $z$ if there is a non-decreasing function $\varphi_z:(0,\infty)\to(0,\infty)$ and $r_z>0$ such that
$u_{\chi_{B(z,r)}}^\vee\ge 1$ on $B(z,\varphi_z(r))$ for every $r \in (0, r_z)$ with $P(B(z,r),X)<\infty$.

Now, suppose that $\Hcal(\dOm)<\infty$, $I\subset\dOm$, and that $\dOm$ has positive mean curvature at 
each $z\in I$, and suppose that 
$f\in \BV_\loc(X)\cap C(X)$. Then, there is a weak solution $u\in\BV_\loc(X)$ to the Dirichlet problem in $\Om$ with boundary data $f$
such that for all $z\in I$,
\[
\lim_{\Om\ni y\to z}u(y)=f(z).
\]
Note that the planar domain 
\[
\Om=\{(x,y)\in\R^2\, :\, x^2+y^2<1,\text{ and }|y|>x^2\text{ when }x>0\}
\]
has the property that $\dOm$ has positive mean curvature at every $z\in \dOm\setminus\{(0,0)\}$. Hence, even though $\dOm$ does not 
have positive mean curvature, the conclusion of Theorem~\ref{thm:main} applies to each point in $\dOm\setminus \{(0,0)\}$.
On the other hand, if $\dOm$ is not of positive mean curvature at some $z\in \dOm$, then it is possible to
find a Lipschitz function $f$ on $X$ and a weak solution $u$ to the Dirichlet problem on $\Om$ with boundary data $f$ such that
$\lim_{\Om\ni y\to z}u^\wedge(y)$ either does not exist or 
is different from $f(z)$. Thus, positive mean curvature of $\dOm$ at a point $z\in \dOm$ determines 
whether that point is a regular point or not.
\end{remark}
\begin{remark}
Given a Lipschitz function $f$ defined on $\dOm$, we can apply the McShane extension theorem and then use 
Theorem~\ref{thm:main} to obtain a solution to the Dirichlet problem in $\Om$ with boundary data $f$.

It is in fact possible to further relax the assumptions on the boundary data \emph{provided} that $\Om$ satisfies some 
further geometric conditions. For such domains, we will show in the next section that
given $f\in \Ccal(\dOm)$,  one can apply the results of~\cite{MSS} to construct a bounded continuous BV extension 
of $f$ to $X$ so that Theorem~\ref{thm:main} may be used.
\end{remark}
\begin{prop}
\label{pro:strong-sol=weak-sol}
Under the assumptions of Theorem~\ref{thm:main}, every weak solution to the Dirichlet problem in 
$\Om$ with boundary data $f$ is a solution to the said problem,  and
conversely, every 
solution, when extended by $f$ outside $\Om$, is a weak solution.
\end{prop}
\begin{proof}
Let $F_t$, $K_t$, and $u$ be as in the proof of 
Theorem~\ref{thm:main}.

Let $w$ be a weak solution to the Dirichlet problem in $\Om$ with boundary data $f$. 
Then, $\|Dw\|(\overline{\Om}) = \|Du\|(\overline\Om)$. Define $G_t=\{x\in X : w(x)>t\}$ for $t\in\R$. 
Then, $\chi_{G_t} = \chi_{F_t}$ in $X \setminus \Om$ for every $t\in \R$. In particular, 
$P(K_t, \overline{\Om}) \le P(G_t, \overline{\Om})$ for a.e.\@ $t\in \R$ since 
$K_t$ is a minimal weak solution set for the boundary data $\chi_{F_t}$  for a.e.\@ $t\in\R$,
as seen from the discussion following~\eqref{eq:Kt-def}.
By the coarea formula, we have
\[
  \int_\R P(G_t, \overline\Om)\,dt = \|Dw\|(\overline{\Om}) = \|Du\|(\overline\Om) = \int_\R P(K_t, \overline\Om)\,dt. 
\]
Consequently, $P(G_t, \overline\Om) = P(K_t, \overline\Om)$ for a.e.\@ $t\in \R$. Hence, $G_t$ is 
a weak solution set for the boundary data $\chi_{F_t}$ for all such $t\in\R$. Observe also that 
\[
P(G_t, \Om) = P(G_t, \overline\Om) = P(K_t, \overline\Om) = P(K_t, \Om)
\]
for a.e.\@ $t\in \R$ 
by Proposition~\ref{pro:weak-to-strong-positiveCurve} together with
Lemma~\ref{lem:ext-of-trace-global}. In particular, invoking the coarea formula yields that 
$\|Dw\|(\Om) = \|Du\|(\Om)$.

Next, 
for $x\in \dOm$, if $t\in\R$ such that $f(x)>t$, then there is some $r>0$ such that $B(x,r)\subset F_t$.
Then, $B(x,\varphi(r))\cap\Om\ssub G_t$ by Lemma~\ref{lem:Not-Close-to-bdry}, which allows us to conclude that 
$Tw(x)\ge t$. It follows that $Tw\ge f$ on $\dOm$ and in fact, 
\[
\liminf_{\Om\ni y\to x}w^\wedge(y)\ge f(x)
\] 
for every $x\in \dOm$.
Reverse inequality follows in a similar manner. 
Therefore, $w$ is a solution to the Dirichlet problem for boundary data $f$ since $Tw = Tu = f$ in 
$\dOm$ and $\|Dw\|(\Om) = \|Du\|(\Om)$ as shown above, while $u$ is a solution to the said problem.

Finally, let $v\in\BV(\Om)$ be a solution to the Dirichlet problem in $\Om$ with boundary data $f$. 
Let $\tilde{v}$ be defined as the extension of $v$ to $X$ by setting it equal to $f$ outside $\Om$. By 
Lemma~\ref{lem:ext-of-trace-global}, we see that 
$P(\{\tilde{v} > t\}, \overline\Om) = P(\{v > t\}, \Om)$ for a.e.\@ $t\in\R$. The coarea formula then yields that $\tilde{v}\in \BV_\loc(X)$ and
\[
  \|D\tilde{v}\|(\overline\Om) = \int_\R P(\{\tilde{v} > t\}, \overline\Om) =  \int_\R P(\{v > t\}, \Om)\,dt = \|Dv\|(\Om).
\]
Since $u$ is a solution to the Dirichlet problem with boundary data $f$, we obtain that
\[
  \|D\tilde{v}\|(\overline\Om) = \|Dv\|(\Om) = \|Du\|(\Om) \le  \|Du\|(\overline\Om).
\]
As $u$ is also a weak solution to the Dirichlet problem with boundary data $f$, it follows that so is $\tilde{v}$.
\end{proof}
\begin{remark}
In~\cite{KLLS}, the following modified minimization problem was studied. Given a Lipschitz function $f$ with a compact support in $X$, 
the goal there was to find a function $\overline{u}\in \BV(\Om)$ such that
\[
J_+(\overline{u})\coloneq \| D\overline{u}\|(\Om)+\int_{\dOm}|T\overline{u}-f|\, dP_+(\Om,\cdot)
  \le J_+(v)
\]
for all $v\in \BV(\Om)$. 
If the domain $\Om$ has finite perimeter and satisfies an exterior measure density condition (that is,
$\limsup_{r\to 0^+}\mu(B(x,r)\setminus\Om)/\mu(B(x,r))>0$ for $\Hcal$-a.e.~$x\in\dOm$),
as well as all three conditions required for existence of a 
bounded trace operator as listed on page \pageref{enum:conditions for trace}, then the desired function 
$\overline{u} \in \BV(\Om)$ can be constructed by solving the Dirichlet problem for $p$-energy 
minimizers on $\Om$ and then letting $p\to 1^+$, see~\cite[Theorem~7.7]{KLLS}.
In fact, the solution obtained this way belongs to the global class $\BV(X)$ with 
$\overline{u}=f$ on $X\setminus\Om$.
The functional $J_+$ is related to the functional $J$ defined in~\eqref{eq:p-to-1-Jfunct}, but unlike there, 
the Radon measure $P_+$ associated with $J_+$ is the \emph{internal} perimeter measure of $\Om$. 
It was shown in \cite[Theorem~6.9]{KLLS} that $P(\Om, \cdot) \le P_+(\Om, \cdot) \le C P(\Om, \cdot)$ 
for some $C\ge 1$. 
If, in addition to the above conditions on $\Om$, the boundary $\dOm$ has positive mean curvature, 
then we can use Theorem~\ref{thm:main} to find
a weak solution $u$ to the Dirichlet problem in $\Om$ with boundary data $f$. 
Then, by~\cite[Proposition~7.5]{KLLS},
\begin{align*}
\| D\overline{u}\|(\overline{\Om})=\| D\overline{u}\|(\Om)+\int_{\dOm}|T\overline{u}-f|\, dP(\Om,\cdot)
   \le J_+(\overline{u})\le J_+(u)&=\| Du\|(\Om)\\
    &=\| Du\|(\overline{\Om}).
\end{align*}
It follows that $\overline{u}$ is a weak solution to the Dirichlet problem in $\Om$ with boundary
data $f$. 
Subsequently, by Proposition~\ref{pro:strong-sol=weak-sol}, $\overline{u}$ is a (strong) solution as well, that is, 
$T\overline{u}=f$. Therefore,
\[
J_+(u)=\| Du\|(\Om)=\| D\overline{u}\|(\Om)=J_+(\overline{u}),
\]
and so it follows that a (strong) solution to the Dirichlet problem in $\Om$ with boundary data
$f$ is also a minimizer of the functional $J_+$.

In conclusion, the class of weak solutions, the class of strong solutions, and the
class of minimizers of the functional $J_+$ coincide for domains $\Om$ that satisfy all the hypotheses given above.
\end{remark}
\section{General continuous boundary data}
\label{sec:cont-data}
The main theorem of the paper, Theorem~\ref{thm:main},
assumes that the boundary data is given as a restriction of a globally continuous $\BV(X)$
function to $\dOm$. In this section, we will prove that under certain circumstances,
every $f \in \Ccal(\dOm)$ can be extended to a
globally continuous $\BV$ function in the whole space $X$, and hence 
Theorem~\ref{thm:main} applies in such a case as well.

To this end, we will slightly modify a construction from~\cite{MSS} to find a $\BV$ extension.

\noindent{\bf Further assumptions on $\Om$:} In order
to obtain a bound on the total variation of the extended function,
one needs to assume that $\Hcal(\dOm)<\infty$, $\Hcal\lfloor_{\dOm}$ is doubling on $\dOm$,
and that the codimension~$1$
Hausdorff measure on $\dOm$ is \emph{lower codimension~$1$ Ahlfors regular}, i.e., there is $C > 0$ such that
\begin{equation}
\label{eq:lower-codim1-reg}
\Hcal(B(x,r)\cap \dOm) \ge C \,\frac{\mu(B(x,r))}{r}
\end{equation}
for every $x\in\dOm$ and $0<r<2 \diam(\dOm)$. On the other hand, apart from the last theorem of this section, we 
do not need the assumption $\mu(X\setminus\Om)>0$ from the list of standing assumptions given at the 
beginning of Section~\ref{sec:prelresults}.

The paper~\cite{MSS} further assumes that a localized converse of \eqref{eq:lower-codim1-reg} holds 
true and that $\mu$ satisfies a local measure density
property, i.e., $\mu(B \cap \Om) \ge C \mu(B)$ whenever $B$ has center in $\overline{\Om}$. 
These two properties are however
used only to prove that the trace of the extended function coincides with the given boundary data.
Since we only deal with continuous functions $f$, we
will prove directly that the extended function is continuous in $X$,
and so we can get by without these additional assumptions.

Given a set $Z\subset X$ and a (locally) Lipschitz function $f: Z \to \R$,  we define
\[
\LIP (f,Z)=\sup_{x,y\in Z\colon x\ne y}\frac{|f(y)-f(x)|}{d(y,x)}.
\]
When $x$ is a point in the interior of $Z\subset X$, we set
\[
\Lip f(x)=\limsup_{y\to x}\frac{|f(y)-f(x)|}{d(y,x)}.
\]
Note that if $f$ is a (locally) Lipschitz function on $X$, then $\Lip f$ is an upper gradient of $f$; 
see for example~\cite{Hei01}.
In particular, $\| Df\|(X)\le \| \Lip f \|_{L^1(X)}$.

By \cite[Proposition~4.1.15]{HKST}, there is a countable collection $\mathcal{W} = \{B(p_{j,i}, r_{j,i})\}$ of balls in
$X \setminus \dOm$ so that
\begin{itemize}
	\setlength{\parskip}{0.5pt plus 0.5pt minus 0.5pt}
	\setlength{\itemsep}{0.5pt plus 1pt minus 0.5pt}
	\item $\bigcup_{j,i} B_{j,i} = X \setminus \dOm$,
	\item $\sum_{j,i} \chi_{2 B_{j,i}} \le C$,
	\item $2^{j-1} < r_{j,i} \le 2^j$ for all $i$, and
	\item $r_{j,i} = \frac18 \dist(p_{j,i}, \dOm)$,
\end{itemize}
where the constant $C$ depends solely on the doubling constant of $\mu$.

By \cite[Theorem~4.1.21]{HKST}, there is a Lipschitz partition of unity subordinate to the Whitney decomposition
$\mathcal{W}$, that is, $\sum_{j,i} \phi_{j,i} = \chi_{X \setminus \dOm}$,  $0 \le \phi_{j,i} \le \chi_{2 B_{j,i}}$, and $\phi_{j,i}$ is $C/{r_{j,i}}$-Lipschitz continuous.

Let $f:\dOm\to\R$ be a Lipschitz continuous function.
Given the center of a Whitney ball $p_{j,i} \in X \setminus \dOm$, we choose a closest point $q_{j,i} \in \dOm$
and define $U_{j,i} = B(q_{j,i}, r_{j,i}) \cap \dOm$. Then, we define a linear extension $Ef$ by setting
\[
Ef(x) = \sum_{j,i}
\biggl( \fint_{U_{j,i}} f(y)\,d\Hcal(y)\biggr) \phi_{j,i} (x), \quad x \in X \setminus \dOm.
\]

We can now proceed as in \cite[Section~4]{MSS} and build
up a (non-linear) extension for general continuous boundary data.

Since $f \in \Ccal(\dOm)$, there is a sequence of Lipschitz continuous functions $\{f_k\}_{k=1}^\infty$ such that
$\| f_k-f\|_{L^\infty(\dOm)}<2^{-k}$ for $k>1$  (by the Stone--Weierstrass theorem as $\dOm$ is compact) and
$\|f_{k+1}-f_k\|_{L^1(\dOm)} \le 2^{2-k} \|f\|_{L^1(\dOm)}$. For technical reasons, we choose $f_1 \equiv 0$.
Then, we pick a decreasing sequence of real numbers $\{\rho_k\}_{k=1}^\infty$ such that:
\begin{itemize}
	\setlength{\parskip}{0.5pt plus 0.5pt minus 0.5pt}
	\setlength{\itemsep}{0.5pt plus 1pt minus 0.5pt}
	\item $\rho_1 < \diam(\Om)/2$,
	\item $0< \rho_{k+1} \le \rho_k /2$, and
	\item $\sum_k \rho_k \LIP(f_{k+1}, \dOm) \le C \|f\|_{L^1(\dOm)}$.
\end{itemize}
This sequence of numbers can be used to define layers in $X \setminus \dOm$. Let
\[
\psi_k(x) = \max\biggl\{0, \min\biggl\{1, \frac{\rho_k - \dist(x, \dOm)}{\rho_k - \rho_{k+1}} \biggr\} \biggr\}, \quad x \in X \setminus \dOm.
\]
Then, we define 
\[ 
\Ext f(x) = \begin{cases} \sum_{k=2}^\infty \bigl( \psi_{k-1}(x) - \psi_k(x)\bigr) E f_k(x) &\text{ when }x\in X \setminus \dOm,\\
  f(x)&\text{ when }x\in\dOm.\end{cases}
\] 
Note that $\supt(\psi_{k-1}-\psi_k)= \{x\in X\, :\, \rho_{k+1}\le \text{dist}(x,\dOm)\le \rho_{k-1}\}$,
and
\[
\sum_{k=2}^n \bigl(\psi_{k-1}(x)-\psi_k(x)\bigr)=\psi_1(x)-\psi_n(x) \to \psi_1(x)
\]
for every $x\in X \setminus \dOm$ as $n\to\infty$.
\begin{lem}
	\label{lem:nonlin-ext-continuous}
	Let $f \in \Ccal(\dOm)$ and $z \in \dOm$. Then,
	\[
	\lim_{X\setminus\dOm\ni x\to z}\Ext f(x)=f(z).
	\]
\end{lem}
\begin{proof}
Fix $z\in\dOm$ and $\eps>0$.
For $m\in\N$ with $m\ge 2$ and $x\in X\setminus\Om$ with $\text{dist}(x,\dOm)<\rho_m$, we see that
\begin{align*}
  |\Ext f(x)-f(z)|&\le \bigg| \sum_{k=m}^\infty \bigl(\psi_{k-1}(x)-\psi_k(x)\bigr)\bigl(E f_k(x)-f(z)\bigr)\bigg|\\
  &\le \sum_{k=m}^\infty \bigl(\psi_{k-1}(x)-\psi_k(x)\bigr)\bigl|E f_k(x)-f(z)\bigr|\,.
\end{align*}

Suppose that $x \in 2B_{j,i}$ for some ball $B_{j,i}=B(p_{j,i}, r_{j,i}) \in \mathcal{W}$ with $q_{j,i}$ being a 
closest point to $p_{j,i}$ in $\dOm$. Then,
	\[
	8 r_{j,i} = d(p_{j,i}, q_{j,i}) = \dist(p_{j,i}, \dOm) \le d(p_{j,i}, z) \le d(p_{j,i}, x) + d(x, z) < 2r_{j,i} + d(x, z).
	\]
	Hence $r_{j,i} < \frac16 d(x,z)$. Thus,
	\[
	d(z, q_{j,i}) \le d(z, x) + d(x, p_{j,i}) + d(p_{j,i}, q_{j,i}) < d(z,x) + 2 r_{j,i} + 8 r_{j,i} < \frac{8}{3} d(z,x).
	\]
Consequently, every $y \in U_{j,i} = B(q_{j,i}, r_{j,i})\cap \dOm$ satisfies $d(z, y) \le \frac{17}{6} d(z,x)$. 

As $f$ is continuous, there is $\delta>0$ such that $|f(y)-f(z)|<\eps$ whenever $y\in\dOm$ with $d(z,y)<\delta$.
In particular, if $x\in 2B_{j,i}$ and $d(z,x) < \frac{6}{17} \delta$, then $|f(y) - f(z)| < \eps$ whenever $y\in U_{j,i}$.
Hence, we obtain for every $x \in B(z, \frac{6}{17}\delta) \setminus \dOm$ that
	\begin{align*}
	|Ef_k(x)-f(z)|&\le \sum_{{j,i}}\fint_{U_{j,i}}|f_k(y)-f(z)|\, d\Hcal(y)\, \phi_{j,i}(x)\\
	&\le \sum_{{j,i}}\fint_{U_{j,i}}\bigl(|f_k(y)-f(y)|+|f(y)-f(z)|\bigr)\, d\Hcal(y)\, \phi_{j,i}(x)\\
	&\le \sum_{{j,i}}\bigl(\| f_k-f\|_{L^\infty(\dOm)}+\eps\bigr)\, \phi_{j,i}(x) \\
	&= \| f_k-f\|_{L^\infty(\dOm)}+\eps.
	\end{align*}
	Therefore, if $d(x,z)<\min\{\rho_m, \tfrac{6}{17}\delta\}$, we have that
	\begin{align*}
	|\Ext f(x)-f(z)|&\le \sum_{k=m}^\infty \bigl(\psi_{k-1}(x)-\psi_k(x)\bigr)\bigl(\| f_k-f\|_{L^\infty(\dOm)}+\eps\bigr)\\
	&\le \sup_{j\ge m}\bigl(\| f_j-f\|_{L^\infty(\dOm)}+\eps\bigr)\sum_{k=m}^\infty \bigl(\psi_{k-1}(x)-\psi_k(x)\bigr)\\
	&\le \sup_{j\ge m} \| f_j-f\|_{L^\infty(\dOm)}+\eps\,.
	\end{align*}
	Choosing $m > 1$ such that $2^{-m}<\eps$ then yields for
	$x\in B(z,\min\{\rho_m, \tfrac{6}{17}\delta\})\setminus\dOm$ that
	\[
	|\Ext f(x)-f(z)|\le \sup_{j\ge m} \| f_j-f\|_{L^\infty(\dOm)}+\eps < 2^{-m} + \eps <2\eps,
	\]
	which completes the proof.
\end{proof}
\begin{prop}
	\label{pro:cont-BV-extension}
	For $f \in \Ccal(\dOm)$, we have
	$\Ext f \in \Ccal(X) \cap \BV(X)$. Moreover, $\Ext f$ is compactly supported and
	\begin{align*}
	\|\Ext f\|_{L^\infty(X)} & \le \|f\|_{L^\infty(\dOm)} + 1
	\quad\text{and}\\
	\| D\Ext f\|(X) & \le C \bigl(1+\Hcal(\dOm)\bigr) \bigl( \|f\|_{L^1(\dOm)} + \|f\|_{L^\infty(\dOm)}+1\bigr).
	\end{align*}
\end{prop}
\begin{proof}
	$\Ext f$ is
	locally Lipschitz in $X\setminus \dOm$ by its construction. Lemma~\ref{lem:nonlin-ext-continuous}
	shows that $\Ext f$ is continuous on $X$. 
	The fact that $\Ext f$ is compactly supported follows from $\supt \Ext f \subset \supt \psi_1$,
	which is bounded and hence compact.
	The estimate $\|\Ext f\|_{L^\infty(X)} \le \|f\|_{L^\infty(\dOm)} + 1$ follows directly from
	the definition of $\Ext f$ together with the requirements on the functions $f_k$ that went into its definition.
	
	It follows from \cite[Proposition~4.2]{MSS} that $\Ext f \in \BV(X \setminus \dOm)$ with the estimate
	\begin{equation}\label{eq:Lip Ext f estimate}
	\|\Lip (\Ext f)\|_{L^1(X\setminus \dOm)} \le C \bigl(1+\Hcal(\dOm)\bigr) \|f\|_{L^1(\dOm)}.
	\end{equation}
	Fix $n \in \N$. We can cover $\dOm$ by finitely many balls $\{D_\ell: \ell=1,2,\ldots \}$ of radii
	$\rho_\ell < \frac1n$ so that
	\[
	\sum_\ell \frac{\mu(D_\ell)}{\rho_\ell} < \Hcal(\dOm) + \frac1n.
	\]
	Let
	\[
	\eta_n(x) = \min\biggl\{ 1, \frac{\dist(x, D_\ell)}{\rho_\ell} \colon \ell = 1,2,\ldots \biggr\}, \quad x\in X.
	\]
	Then,
	\[
	\Lip \eta_n \le \sum_\ell \frac{1}{\rho_\ell}\chi_{2\overline{D}_\ell \setminus D_\ell}.
	\]
	Set $F_n = \eta_n \Ext f$. Since $F_n = 0$ in a neighborhood
	of $\dOm$, it follows that $F_n$ is Lipschitz continuous. The Leibniz rule for (locally) Lipschitz functions yields that
	\begin{align*}
	\Lip F_n  
	&\le \eta_n \Lip (\Ext f) + |\Ext f| \Lip \eta_n\\
	& \le \chi_{X \setminus \bigcup_\ell D_\ell} \Lip (\Ext f)
	+  (\|f\|_{L^\infty(\dOm)} + 1)\sum_\ell  \frac{\chi_{2\overline{D}_\ell \setminus D_\ell}}{\rho_\ell}.
	\end{align*}
	Thus
	\begin{align*}
	\|DF_n\|(X)
	&\le \int_X \Lip F_n\,d\mu\\
	& \le \| \Lip (\Ext f)\|_{L^1(X\setminus\dOm)}
	+ (\|f\|_{L^\infty(\dOm)} + 1)\sum_\ell \frac{\mu(2\overline{D}_\ell \setminus D_\ell)}{\rho_\ell} \\
	& \le C \bigl(1+\Hcal(\dOm)\bigr) \|f\|_{L^1(\dOm)}
	+ C_D^2 (\|f\|_{L^\infty(\dOm)} + 1) \biggl(\Hcal(\dOm) + \frac 1n\biggr)
	\end{align*}
	by \eqref{eq:Lip Ext f estimate}.
	
	Direct computation shows that  $F_n \to \Ext f$ in $L^1(X)$ as $n \to \infty$. The lower semicontinuity of
	$\BV$ energy as in~\eqref{eq:BV-lsc} then implies that
	\[
	\|D(\Ext f)\|(X) \le \liminf_{n\to\infty} \|D F_n\|(X) \le 
	C \bigl(1+\Hcal(\dOm)\bigr) (\|f\|_{L^1(\dOm)}
	+\|f\|_{L^\infty(\dOm)} + 1).
	\qedhere
	\]
\end{proof}
Recall that we assume $\Om$ and $X$ to satisfy all the standing assumptions listed in Section~\ref{sec:prelresults} and the	further assumptions listed at the beginning of the current section.
\begin{thm}\label{thm:main_cont-data}
	Suppose that $\dOm$ has positive mean curvature.
	Let $f \in \Ccal(\dOm)$. Then, there is a function $u\in \BV(\Om)$ that is a solution to the Dirichlet
	problem in $\Om$ with boundary data $f$.
	Furthermore,
	\[
	\lim_{\Om\ni x\to z}u(x)=f(z)
	\]
	whenever $z\in \dOm$.
\end{thm}
\begin{proof}
	By Proposition~\ref{pro:cont-BV-extension}, there is a bounded function
	$\Ext f \in \Ccal(X) \cap \BV(X)$ such that $f = (\Ext f)\bigr|_{\dOm}$. Hence, we can apply
	Theorem~\ref{thm:main} to the boundary data $\Ext f$.
\end{proof}
\section{Counterexamples}
Unlike in~\cite{SWZ}, solutions to the Dirichlet problem may fail to be continuous even if the boundary data are
Lipschitz continuous. Moreover, uniqueness of the solutions cannot be guaranteed either. In this section we illustrate
these issues with a series of examples in the plane with a weighted Lebesgue measure $d\mu = w \,d\Leb2$ on a
domain $\Omega\subset \R^{2}$. 
The two principal examples are Example~\ref{exa:lightdiamondtight} and
Example~\ref{exa:litedmdheavycore} with continuous weights $w$, the first demonstrating the failure of continuity 
of the solution all the way up to the boundary, 
and the second demonstrating non-uniqueness. To make these two examples easier to visualize, we also 
provide preliminary examples with piecewise constant weights, giving
 simpler illustrations of discontinuity and non-uniqueness.

In the settings considered in this section, sets whose characteristic functions are
of least gradient have boundaries which are shortest
paths with respect to a weighted distance. Hence we first investigate shortest paths. This is the 
aim of the next subsection. Once this is done, we continue on in the subsequent subsection to describe the examples.
\subsection{Minimal perimeter of sets in weighted Euclidean setting, and length with respect to weights}
Suppose $w>0$ is a continuous weight on $\R^{2}$ for which $d\mu = w \,d\Leb2$ is doubling and satisfies a
$1$-Poincar\'e inequality. According to Corollary 2.2.2 and Theorem 3.2.3 of \cite{Cam}, if $E\subset \Om$ is measurable,
then
\[
P_{w}(E,\Om)=\int_{\Om\cap \partial_{m}E} w\, d\Hcal^{1},
\]
where $P_{w}$ indicates perimeter with respect to $\mu$ and $\partial_{m}E$ is the measure theoretic boundary
with respect to $d\Leb{2}$. While some of the weights considered in this section are not continuous, they are piecewise continuous, and the discontinuity set is contained
in a piecewise smooth $1$-dimensional set, where the weight will be lower semicontinuous.
Hence, a simple argument shows that the above-stated conclusion of~\cite{Cam} holds here as well.

For a weight $w$ on $\mathbb{R}^{2}$ 
and a Lipschitz path $\phi\colon [0,1]\to \mathbb{R}^{2}$,
the weighted length of $\phi$ is  
\[
I(\phi,w)\coloneq \int_{0}^{1} |\phi'(t)|w(\phi(t))\,dt. 
\]
By the discussion above, this weighted length is the perimeter measure
of the set whose  boundary is given by the trajectory of $\phi$; note that such $\phi$ is injective	$\Leb1$-a.e.\@ in $[0,1]$.
In this setting, the shortest path is one that minimizes weighted length.

For every open set $G \subset \Om$ where the weight function $w$ is constant and each set $E$ of 
minimal boundary surface, the connected components of $\partial_mE \cap G$ are straight line segments. 
This follows because the shortest paths (with respect to weighted length) inside $G$ are Euclidean 
geodesics. 

We next consider a simple weight.
\begin{exa}[{\bf Ibn Sahl--Snell's law}]\label{ncaseweight}\end{exa}
\vspace{-\topsep}
\begin{wrapfigure}{r}{60mm}
  \vspace{-30pt}
  \hfill
  \includegraphics[width=52mm,height=52mm,page=3]{\figfilename}\ 
\end{wrapfigure}
\noindent
First suppose $w_{1}, w_{2}>0$ and $0<y_{1}<y_{2}$. Let $w$ be a weight function with $w(x,y)=w_{1}$ if 
$-y_{1}\leq y< 0$ while $w(x,y)=w_{2}$ if $-y_{2}< y<-y_{1}$. We seek the shortest path with respect to $\mu$ 
joining $(0,0)$ to a point $(x_{2},-y_{2})$. Since the weight is piecewise constant, the shortest path is the 
concatenation of line segments $L_{1}$ and $L_{2}$ in the regions $-y_{1}<y<0$ and $-y_{2}<y<y_{1}$ respectively. 
If these lines meet the line $y=-y_{1}$ at a point $(x,-y_{1})$, then the weighted length is
\[
 d(x)=w_{1}\sqrt{x^{2}+y_1^2}+w_{2}\sqrt{\bigl(x-x_{2}\bigr)^{2}+\bigl(y_{2}-y_{1}\bigr)^{2}}.
\]
The derivative of this length is then
\[d'(x)=\frac{w_{1}x}{\sqrt{x^{2}+y_{1}^{2}}}+\frac{w_{2}(x-x_{2})}{\sqrt{(x-x_{2})^{2}+(y_{2}-y_{1})^{2}}}.\]
Let $\theta_{1}$ and $\theta_{2}$ be the acute angles $L_{1}$ and $L_{2}$ make with the vertical at $(x,-y_{1})$. 
Let $A_1$ and $A_2$ be the Euclidean lengths of $L_{1}$ and $L_{2}$, that is,
$A_1=\sqrt{x^2+y_1^2}$ and $A_2=\sqrt{(x-x_2)^2+(y_2-y_1)^2}$. Then the equation $d'(x)=0$ gives
\[
\frac{w_{1}A_1\sin(\theta_{1})}{A_1}-\frac{w_{2}A_2\sin(\theta_{2})}{A_2}=0.
\]
Either $\theta_1=\theta_2=0$ or rearranging gives the Ibn Sahl--Snell Law \cite{Ras}:
\[\frac{w_{1}}{w_{2}}=\frac{\sin(\theta_{2})}{\sin(\theta_{1})}.\]

Now suppose $0<y_{1}<y_{2}<\cdots <y_{n}$ and $w$ is a weight with $w(x,y)=w_{k}$ if $-y_{k}<y\leq -y_{k-1}$. 
We seek the shortest path joining $(0,0)$ to a point $(x_{n},-y_{n})$. This is a concatenation of lines 
$L_{1}, \ldots, L_{n}$, where $L_{k}$ is the segment in the region $-y_{k}<y<-y_{k-1}$. Let $\theta_{k}$ be 
the acute angle the line $L_{k}$ makes with the vertical. Applying the previous case gives
\[\frac{\sin(\theta_{k})}{\sin(\theta_{k-1})}=\frac{w_{k-1}}{w_{k}} \qquad \mbox{for }2\leq k\leq n.\]
Hence,
\[
\frac{\sin(\theta_{n})}{\sin(\theta_{1})}=\frac{\sin(\theta_{n})}{\sin(\theta_{n-1})}\cdot
\frac{\sin(\theta_{n-1})}{\sin(\theta_{n-2})}
\cdots
\frac{\sin(\theta_{2})}{\sin(\theta_{1})}
= \frac{w_{n-1}}{w_{n}}\cdot \frac{w_{n-2}}{w_{n-1}}\cdots \frac{w_{1}}{w_{2}}
= \frac{w_{1}}{w_{n}}\,.
\]

The situation is similar for weights which linearly interpolate between two values.
\begin{exa}[{\bf Ibn Sahl--Snell's law for continuous weights}]\label{interpolateweight}
Let us suppose that $0<z_{1}<z_{2}<z_{3}$ and $w$ is a weight with $w(x,y)=w_{1}$ for $-z_{1}<y<0$, 
$w(x,y)=w_{2}$ for $-z_{3}<y<-z_{2}$, and
\[
w(x,y)=w_{1}\left(\frac{y+z_{2}}{z_{2}-z_{1}}\right) - w_{2}\left(\frac{y+z_{1}}{z_{2}-z_{1}}\right)
\]
for $-z_{2}<y<-z_{1}$.
We seek the path $\gamma$ which minimizes the weighted length $I(\phi,w)$, over Lipschitz paths 
$\phi$ with $\phi(0)=0$ and $\phi(1)=(x_{3},-z_{3})$ for a fixed point $x_{3}$.

One can choose a natural family of weights $w_{n}$ of the type considered in Example \ref{ncaseweight}, agreeing 
with $w$ for $y>-z_{1}$ and $y<-z_{2}$, with $w_{n}\to w$ uniformly. Choose Lipschitz curves 
$f_{n}\colon [0,1]\to \mathbb{R}^{2}$ which minimize $I(\phi, w_{n})$ among Lipschitz curves $\phi$ with 
$\phi(0)=0$ and $\phi(1)=(x_{3},-z_{3})$. Then $I(f_{n},w_{n})\leq I(\phi,w_{n})$ for any competitor $\phi$. 
The curves $f_{n}$ converge to a Lipschitz curve $f$ joining the desired points and $I(f, w)\leq I(\phi, w)$ for 
any competitor $\phi$, so $f$ minimizes the weighted length with respect to $w$.

Since $w$ and $w_{n}$ agree and are constant for $y>-z_{1}$ and $y<-z_{2}$, the path $f$ consists of straight 
lines $L_{1}$ and $L_{2}$ in those regions. Applying Example \ref{ncaseweight} to each $w_{n}$ and letting 
$n\to \infty$, we see that if $\theta_{1}$ and $\theta_{2}$ are the acute angles made between the 
vertical and $L_{1}, L_{2}$ respectively, then
\[
\frac{\sin(\theta_{2})}{\sin(\theta_{1})}= \frac{w_{1}}{w_{2}},
\]
regardless of what happens in the region $-z_{2}<y<-z_{1}$.
\end{exa}
\subsection{Examples of discontinuity and nonuniqueness}
We now illustrate the failure of continuity and uniqueness of solutions. In a series of examples we 
consider the domain $\Om = B(0,1)$ in the Euclidean plane $\R^2$ endowed 
with various weighted Lebesgue measures of the form $d\mu = w \,d\Leb2$.

In Examples~\ref{exa:heavydiamond}--\ref{exa:litedmdheavycore}, Dirichlet boundary data 
are defined as $f(x,y) = y+1$, i.e. the vertical distance from the lowest point of $\dOm$.

We applied the main idea of the proof of
Theorem~\ref{thm:main} to construct a function $u: \Om \to [0,2]$ that 
is a solution to the Dirichlet problem in $\Om$ with boundary data $f$.
For each 
$t \in [0,2]$, we constructed a set $E_t$ of minimal boundary surface in $\Omega$ so that 
$\Tr \chi_{E_t}=\chi_{\{f>t\}}$ on $\partial\Omega$. Moreover, we ensured that the sets 
$E_t$ are nested so that $E_t \subset E_s$ whenever $s<t$.  In the proof of 
Theorem~\ref{thm:main}, the desired nesting of the sets $E_t$ was 
achieved by choosing minimal weak solution sets for $\chi_{\{f>t\}}$. 
However, that need not be the only way to obtain the nesting 
to be able to construct a solution (cf.\@ commentary on 
maximal weak solution sets at the end of 
Section~\ref{sec:prelresults}).

The principal part of the work in these examples is to identify 
$\partial E_t$, as $E_t$ is the connected component of $\Om$ 
lying above $\partial E_t$. Since $E_t$ is of minimal 
boundary surface, $\partial E_t$ is the shortest 
path that connects the points of 
$\partial\{z\in\dOm \colon f(z) > t\}$, 
which turns out to be the points of $\dOm$ with $y$-coordinate equal to $t-1$. 
For piecewise constant weights $w$, the superlevel set (which is of minimal boundary surface) has as its 
boundary a concatenation of straight line segments in $\Omega$. The weighted length of this concatenated path is the 
shortest amongst all paths in $\Om$ with the same endpoints. This follows from the Ibn Sahl--Snell law described above.
\subsubsection*{Discontinuous solutions for the least gradient problem:
the Eye of Horus}
\begin{exa}
\label{exa:heavydiamond}
Having fixed a constant $\alpha >1$, let the weight be given by
\[
  w(x,y) =
  \begin{cases}
    \alpha & \text{if } |x|+|y| <  \frac12,\\ 
    1 &  \text{otherwise.}
  \end{cases}
\]
For the sake of brevity, let $K = \{(x,y)\in \Om\colon |x|+|y| < \frac12\}$. 

(a) First, we consider the case when $\alpha \ge 3/\sqrt{5}$.

We start by finding the superlevel set $E_{1}$ of the solution, which corresponds to the value of the boundary data $f(x,y)=y+1$ at $(-1,0)$ (or equivalently at $(1,0)$). The boundary
of this superlevel set is a path obtained as a concatenation of line segments in $\Om$. If the line segment from $(-1,0)$ intersects the left-hand side of
the boundary of $K$ at the point $(t-1/2,t)$ for some $0<t<1/2$, then the path continues inside $K$. By considerations of
symmetry, it is clear that the part of the path inside $K$ is then parallel to the $x$-axis, and then exits $K$ at the point $(1/2-t,t)$ to continue on to the point $(1,0)$. In this case, the weighted length of this path, and hence the perimeter measure of the superlevel set, is
\[
 g(t)=2\sqrt{t^{2}+\left(t+\frac{1}{2}\right)^2}+\alpha(1-2t).
\]
Note that
\[
g'(t)=\frac{4t+1}{ \sqrt{t^{2}+\left(t+\frac{1}{2}\right)^2} }-2\alpha,
\]
and if $\alpha\ge 3/\sqrt{5}$, then $g'(t)<0$ for all $t\in(0, 1/2)$. Thus, one reduces weighted length by moving the point where the path intersects $K$ up to $(0,1/2)$. Similarly, if the path intersects the boundary of $K$ at $(-t-1/2,t)$ for some $t<0$, then the weighted length is reduced by moving the point where the path intersects $K$ down to $(0,-1/2)$.

Hence the boundary of the superlevel set $E_1$ of the solution is given by two paths. Each path is a concatenation of two line segments, one connecting
$(-1,0)$ to $(0,\pm 1/2)$ and the other connecting $(0,\pm 1/2)$ to $(1,0)$. This identifies $\partial E_{1}$. Clearly $\partial E_{1/2}$ and $\partial E_{3/2}$ respectively are the pieces of the line segments $y=-1/2$ and $y=1/2$ inside $\Omega$.
Analysis similar to the above
enables us to identify all the superlevel sets of the solution to the Dirichlet problem with boundary data $f$.

Having identified $\partial E_t$, we can construct the solution by stacking the sets $E_t$ as in Theorem~\ref{thm:main}. See the figure below. The exact value of $\alpha \ge 3/\sqrt 5$ plays no role here.

\noindent%
\begin{minipage}{\linewidth}
\vspace{\topsep}
\center
\includegraphics[height=2.0in,page=4]{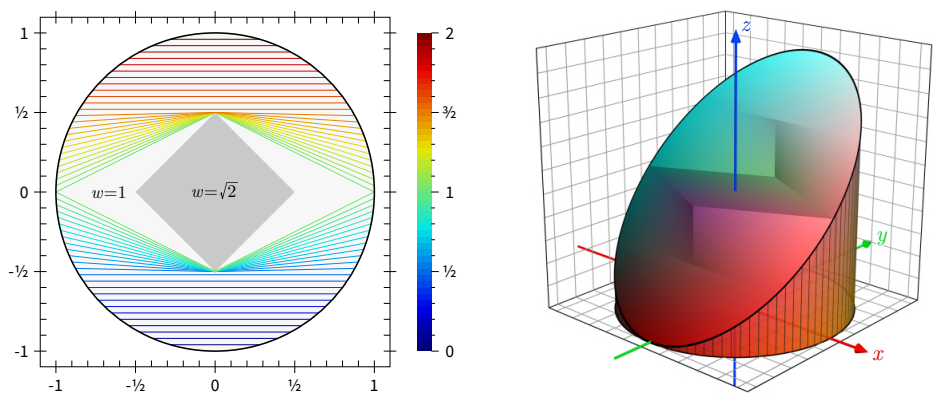}

\captionof{figure}{Level sets $\partial E_t$ for various values of $t \in (0,2)$ in Example~\ref{exa:heavydiamond}\,(a) are shown on the left. A graph of the solution $u$ constructed by stacking the sets $E_t$ as in Theorem~\ref{thm:main} is shown on the right.\label{fig:heavydiamond}}
\vspace{\topsep}
\end{minipage}

(b) Consider now the case where $1 < \alpha < 3/\sqrt{5}$. 

The boundary of the superlevel set with endpoints $(-1,0)$ and $(1,0)$ still consists of two shortest paths, each a concatenation of line segments in $\Om$. Suppose the upper shortest path first meets the diamond at a point $(t-1/2, t)$ for some $0\leq t\leq 1/2$. Since $\alpha < 3/\sqrt{5}$, the function $g'$ above is no longer always decreasing, so the shortest path actually enters the diamond (i.e. $t<1/2$). The angle under which the shortest path enters the diamond is determined by the Ibn Sahl--Snell law, i.e., by the relation $\sin \theta / \sin \frac{\pi}{4} = \alpha$, where $\theta$ is the angle of incidence on the contour of the diamond. Since $\alpha>1$, it follows $\theta>\pi/4$ and hence $t>0$. By symmetry with respect to the $x$-axis, the lower shortest path meets the diamond at a point $(-1/2-t,t)$ with $t<0$. 

Hence, the boundary of the superlevel set $E_1$ consists of two paths, one in the upper half-disk and the other in the lower half-disk, both passing parallel to the $x$-axis while
moving through the central diamond region. These paths forms an oblique hexagonal region in which the solution function is constant. The solution remains continuous in this region, but again exhibits a jump at the top and bottom tips of the central diamond region, see the figure below.

\noindent%
\begin{minipage}{\linewidth}
\vspace{\topsep}
\center
\includegraphics[width=5.8in,height=2.4in,page=5]{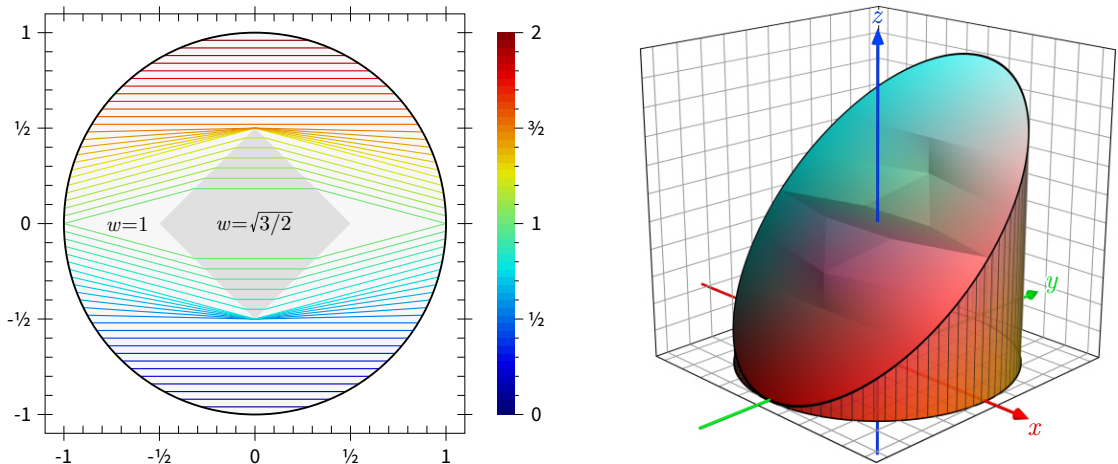}

\captionof{figure}{Level sets $\partial E_t$ for various values of $t \in (0,2)$ in 
Example~\ref{exa:heavydiamond}\,(b), with $\alpha = \sqrt{3/2}$, are shown on the left. 
A graph of the solution $u$ constructed by stacking the sets $E_t$ as in 
Theorem~\ref{thm:main} is shown on the right. Note that the displacement of the level 
sets from the center depends on $\alpha$. \label{fig:heavydiamondB}}
\end{minipage}
\end{exa}
In the example above, the set of points of discontinuity for the solution consisted 
of two points, so had Hausdorff dimension 0. The following example gives a weight for which the set 
of points of discontinuity has Hausdorff dimension 1.
\begin{exa}
\label{exa:heavydisk}
Having fixed a constant $\alpha \ge \pi/2$, let the weight be given by
\[
  w(x,y) =
  \begin{cases}
    \alpha & \text{if } |x|^2+|y|^2 \le \frac14, \\
    1 &  \text{otherwise.}
  \end{cases}
\]
For the sake of brevity, let $K = \{(x,y)\in\Om\colon |x|^2+|y|^2 \le \frac14\}$.

Observe that the choice of $\alpha$ guarantees that the shortest path that connects two points on $\partial K$ is an 
arc lying entirely in $\partial K$. Indeed, let $z_1, z_2 \in \partial K \subset \C$ and $\theta = |\arg \frac{z_1}{z_2}|$. 
Then, the line segment going straight through $K$ that connects these two points has length $2\alpha \sin \frac{\theta}{2}$ whereas the shorter arc in $\partial K$ has length 
$\theta \le 2\alpha \sin \frac{\theta}{2}$ and the inequality is strict unless 
$\theta = 0$ (i.e., $z_1=z_2$) or $\theta = \pi$ (i.e., $z_1 = -z_2$) and $\alpha = \frac{\pi}{2}$.

If $|t-1|\ge\frac12$, then the shortest path connecting the two points of $\{z\in\dOm\colon f(z)=t\}$ is a 
horizontal straight line segment.
If $|t-1| < \frac12$, then the shortest path connecting the two points of $\{z\in\dOm\colon f(z)=t\}$ reaches 
$\partial K$ tangentially, then follows the arc of $\partial K$ to the highest/lowest point of $\partial K$ and then 
continues symmetrically with respect to the axis $x=0$.

\noindent%
\begin{minipage}{\linewidth}
\vspace{2\topsep}
\center
\includegraphics[width=5.8in,height=2.4in,page=6]{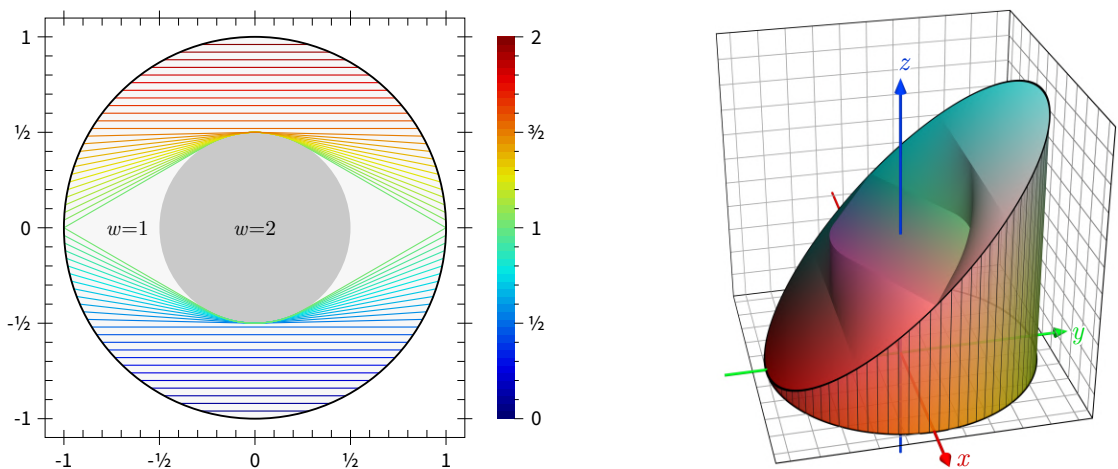}

\captionof{figure}{Level sets $\partial E_t$ for various values of $t \in (0,2)$ in Example~\ref{exa:heavydisk}, with 
$\alpha = 2 > \pi/2$, are shown on the left. A graph of the solution $u$ constructed by stacking the 
sets $E_t$ is shown on the right. \label{fig:heavydisk}}
\end{minipage}
\vspace{0.5\topsep}
\end{exa}
A natural question is whether the discontinuity of solutions could perhaps be caused by the discontinuity of the weight function. 
The next example shows that that is not the case. The following example serves a second purpose as well. Note that
the above examples are of domains where the space is positively curved (for example, in the sense of Alexandrov) inside the domain.
In the next example, the space is negatively curved inside the domain.
\begin{exa}
\label{exa:lightdiamond}
Having fixed a constant $\alpha \in (0, 1)$, let the weight be given by
\[
  w(x,y) =
  \begin{cases}
    \alpha & \text{if } |x|+|y| \le 0.5, \\
    \alpha + \frac{1-\alpha}{0.05} (|x|+|y| - 0.5) & \text{if } 0.5 < |x|+|y| \le 0.55, \\
    1 &  \text{otherwise.}
  \end{cases}
\]
We now check that for $t$ with $t-1>0$ sufficiently small, the $t$-level sets $E_t$ intersect on the central horizontal line of the 
inner diamond. This will result in discontinuity of the corresponding solution on this horizontal segment. Towards this end, 
suppose $\gamma$ is the level set for some $t>1$. Then $\gamma$ is a minimizing curve (with respect to the length 
induced by $w$) joining the two points on $\partial \Om$ with $y$-coordinate $t-1>0$. By symmetry, the portion of 
$\gamma$ in the inner diamond is a horizontal line. We now consider two cases.

If $\gamma$ meets the line $y=0$ then, by symmetry, this first occurs at a point $(x,0)$ with $x\leq -0.5$. It then stays 
horizontal and follows the central horizontal line of the inner diamond until the point $(-x, 0)$, after which $\gamma$ 
increases linearly to meet the unit circle at $y$-coordinate $t-1$.

Suppose $\gamma$ does not meet the line $y=0$. The curve $\gamma$ in the region $x\leq 0, \, y\geq 0$ consists of three 
pieces: a line $L_{1}$ outside the outer diamond $|x|+|y|>0.55$, a line $L_{2}$ inside the inner diamond $|x|+|y|\leq 0.5$, 
and a third piece inside the annulus $0.5<|x|+|y|\leq 0.55$. In the region $x\leq 0, \, y\geq 0$, the weight $w$ is a rotated 
copy of a weight $w_{0}$ from Example \ref{interpolateweight}. The curve $\gamma$ is also minimal with respect to the 
length induced by $w_{0}$. Let $\theta$ be the acute angle $L_{1}$ makes with the direction $(-1,1)$. Since $L_{2}$ is 
horizontal, we know $L_{2}$ makes acute angle $\pi / 4$ with the direction $(-1,1)$. The discussion of 
Example~\ref{interpolateweight} implies that
\[\frac{\sin(\theta)}{\sin(\pi/4)}=\alpha.\]
Hence $\sin(\theta)=\alpha/\sqrt{2}$. If $\alpha<1$ then $\theta$ is bounded away from $\pi/4$. Hence 
$L_{1}$ intersects $\partial \Om$ at a point whose $y$-coordinate is bounded away from $0$, so $t-1$ is 
bounded away from $0$ (with a bound depending only on $\alpha$).

Hence if $t-1>0$ is sufficiently small, the first case applies and $\gamma$ must follow the central horizontal line of the inner diamond.

\noindent%
\begin{minipage}{\linewidth}
\vspace{\topsep}
\center
\includegraphics[width=5.8in,height=2.4in,page=7]{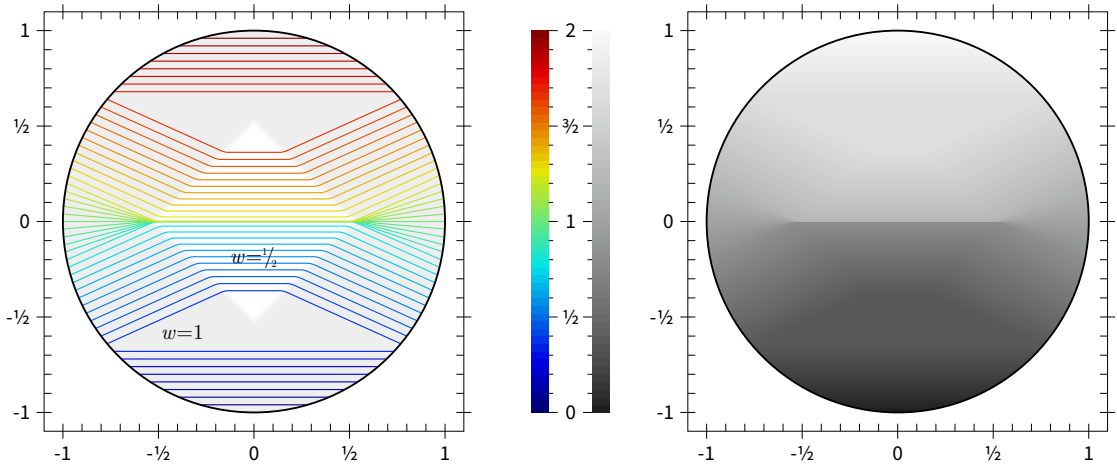}

\captionof{figure}{Level sets $\partial E_t$ for values of $t \in (0,2)$ in Example~\ref{exa:lightdiamond}, with $\alpha = 1/2$, are on the left. A heightmap of the solution $u$ constructed by stacking the sets $E_t$ is on the right, where the grayscale intensity represents the solution's function value. 
\label{fig:lightdiamond}}
\end{minipage}

\vspace{0.5cm}
\noindent%
\begin{minipage}{\linewidth}
\vspace{\topsep}
\center
\includegraphics[width=5.8in,height=2.4in,page=8]{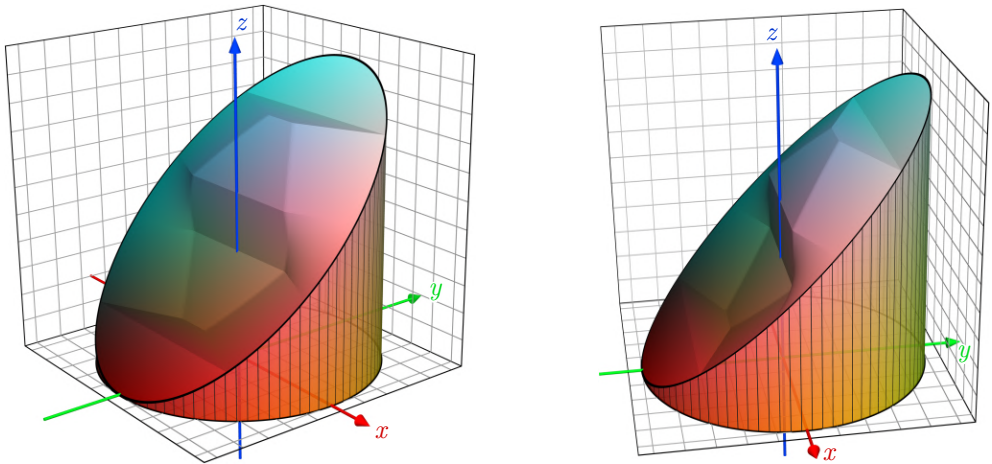}

\captionof{figure}{Graph of the solution $u$ of the Dirichlet problem in 
Example~\ref{exa:lightdiamond}, depicted as a surface $z = u(x, y)$. \label{fig:lightdiamond3D}}
\end{minipage}
\end{exa}
In all the examples above, the solutions are discontinuous inside the domain away from the boundary. 
One might therefore ask whether solutions to the Dirichlet problem are continuous at the boundary. 
The following example shows the set of points of discontinuity can, in fact, reach all the way to the boundary.
\begin{exa}
\label{exa:lightdiamondtight}
Having fixed a constant $\alpha \in (0, 1)$, let the weight be given by
\[
  w(x,y) =
  \begin{cases}
    \alpha + (1-\alpha) (|x|+|y|) & \text{if } |x|+|y| < 1, \\
    1 &  \text{otherwise.}
  \end{cases}
\]
It is not difficult to see that as the weight $w$ decreases as one moves into the disk, this weighted domain
has boundary of positive mean curvature.
Assume $\alpha=1/2$ to simplify the calculations, and fix a discretization scale $n\in \mathbb{N}$. For $1\leq k\leq n$, define
\[A_{k}=\left\{(x,y)\colon (k-1)/n \leq |x|+|y|< k/n\right\}.\] 
We consider the approximating weight
\[
  w_{n}(x,y) =
  \begin{cases}
   \frac{1}{2}(1+\frac{k}{n}) & \text{if }(x,y)\in A_{k} \mbox{ for some } 1\leq k\leq n,\\
    1 &  \text{otherwise.}
  \end{cases}
\]
Suppose $\gamma$ is the boundary of a superlevel set for a solution to the Dirichlet problem with weight $w_{n}$. 
Then $\gamma$ is distance minimizing with respect to the length distance induced by $w_{n}$. Suppose 
$\gamma$ leaves the $x$-axis at $(k_{0}/n,0)$ and intersects the set $\{(x,y):|x|+|y|=1\}$ at a point above the 
$x$-axis. Then the trajectory of $\gamma$ between these points is governed by Snell's law, as discussed in 
Example~\ref{ncaseweight}. Let
\[
w_{0}=\frac{1}{2}\left(1+\frac{k_{0}}{n}\right), \quad w_{k}=\frac{1}{2}\left(1+\frac{k}{n}\right), \quad \theta_{0}=\frac{\pi}{4},
\]
and $\theta_{k}$ be the angle $\gamma$ makes with the line $y=x$ in $A_{k}$. Then, 
$\frac{\sin (\theta_{k})}{\sin (\theta_{0})}=\frac{w_0}{w_k}$ and so 
\begin{equation}\label{sineq}
\sin (\theta_{k})=\frac{1}{\sqrt{2}}\left(\frac{n+k_{0}}{n+k}\right).
\end{equation}
The length of the line $y=x$ inside $A_{k}$ is $1/(n\sqrt{2})$. Hence the length $a_{k}$ of $\gamma$ in 
$A_{k}$ satisfies $a_{k}\cos(\theta_{k})=1/(n\sqrt{2})$. It follows that $h_{k}$, the increase in the $y$-coordinate 
of $\gamma$ in $A_{k}$, is given by
\[h_{k}=a_{k}\sin\left(\frac{\pi}{4}-\theta_{k}\right)=\frac{1-\tan(\theta_{k})}{2n}.\]
Using \eqref{sineq}, this gives
\[h_{k}=\frac{1}{2n}\left(1-\frac{(n+k_{0})}{\sqrt{2(n+k)^{2}-(n+k_{0})^{2}}}\right).\]
Adding these contributions, the total gain in height by $\gamma$ before it leaves the region $|x|+|y|\leq 1$ is
\[
\sum_{k=k_{0}+1}^{n}\frac{1}{2n}\left(1-\frac{(1+\frac{k_{0}}{n})}{\sqrt{2(1+\frac{k}{n})^{2}-(1+\frac{k_{0}}{n})^{2}}}\right).
\]
Given $0<t_{0}<1$, we may choose $k_{0}$ for each $n$ so that $k_{0}/n\to t_{0}$. Then the above sum converges to
\[
H(t_{0})=\frac{1}{2}\int_{t_{0}}^{1} 1 - \frac{1+t_{0}}{\sqrt{2(1+t)^{2}-(1+t_{0})^{2}}} dt.
\]
Since the integrand is strictly positive for $t>t_{0}$, we have $H(t_{0})>0$ whenever $t_{0}<1$. The number 
$H(t_{0})$ gives the $y$-coordinate of the intersection of the set $\{(x,y):|x|+|y|=1\}$ with the curve that is length 
minimizing with respect to $w$ and leaves the $x$-axis at $(t_{0},0)$, moving upwards. Such curves correspond to 
boundaries of superlevel sets of solutions to the Dirichlet problem with respect to $w$, for levels greater than 
$H(t_{0})$. By symmetry, the $y$-coordinate of the intersection point if the curve leaves the $x$-axis at $(t_{0},0)$ 
and goes downwards is $-H(t_{0})$. These correspond to boundaries of superlevel sets of solutions to the same 
Dirichlet problem, for levels less than $-H(t_{0})$.

If we approach $(t_{0},0)$ from above then the values of the solution are greater than $H(t_{0})$, while if we approach 
$(t_{0},0)$ from below the values of the solution are less than $-H(t_{0})$. This implies that the solution is discontinuous 
at $(t_{0},0)$, for any $0<t_{0}<1$. Hence the set of points of discontinuity reach all the way to the boundary.

\noindent%
\begin{minipage}{\linewidth}
\vspace{\topsep}
\center
\includegraphics[width=5.4in,page=9]{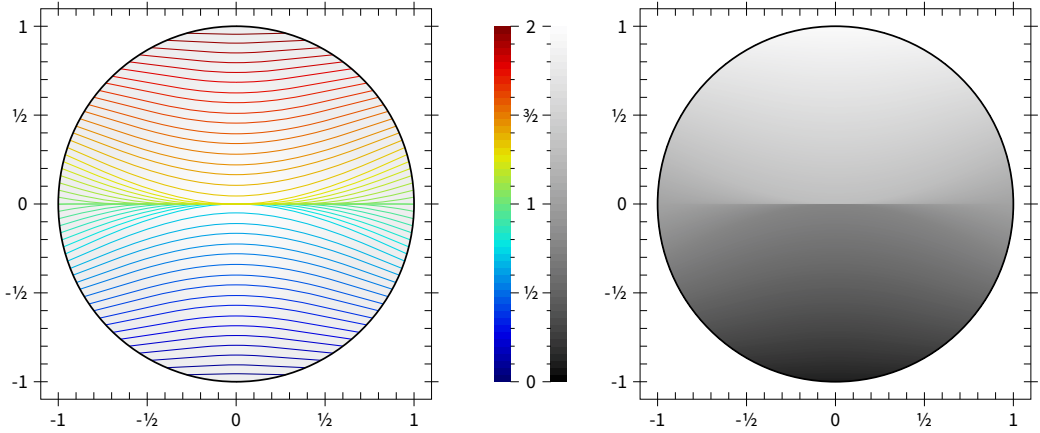}

\captionof{figure}{Level sets $\partial E_t$ for various values of $t \in (0,2)$ in Example~\ref{exa:lightdiamondtight}, with 
$\alpha = 1/2$, are shown on the left. A heightmap of the solution $u$ constructed by stacking the sets $E_t$ is shown 
on the right.
\label{fig:lightdiamondtight}}
\end{minipage}

\noindent%
\begin{minipage}{\linewidth}
\vspace{\topsep}
\center
\includegraphics[width=5.8in,height=69mm,page=10]{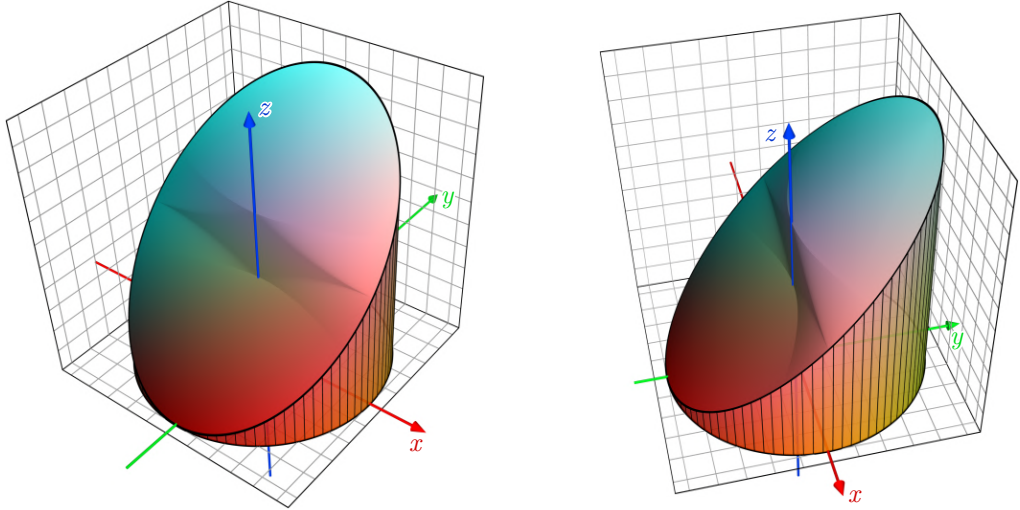}

\captionof{figure}{Graph of the solution $u$ of the Dirichlet problem in Example~\ref{exa:lightdiamondtight}, depicted as a surface $z = u(x, y)$. Observe that the jump-set of the solution lies upon the $x$-axis, i.e., $S_u = \{(x,y)\in \Om: y=0\}$. \label{fig:lightdiamondtight3D}}
\end{minipage}
\vspace{\topsep}
\end{exa}
\subsubsection*{Non-uniqueness of solutions: Third Eye}
The final two examples demonstrate that solutions may fail to be unique. In the first of these two examples,
the space is positively curved in some points inside the domain, and flat at other points in the domain. However, the weight
is not a continuous function. The last example of this paper gives a continuous weight; in this example, the space is 
negatively curved at some points (for example, in $K_{\text{ann}}$), positively curved at some points (for example, in
$K_{\text{in}}$), and flat at some points (for example, in $K_{\text{out}}$).
\begin{exa}
\label{exa:3heavydiamonds}
Having fixed a constant $\alpha \ge \sqrt 2$, let the weight be given by
\[
  w(x,y) =
  \begin{cases}
    \alpha & \text{if } \min\bigl\{\bigl|x-\frac12\bigr|, \bigl|x+\frac12\bigr|\bigr\} 
      +|y| \le \frac14 \quad \text{or}\quad |x|+\bigl|y-\frac14 \bigr| \le \frac18, \\
    1 &  \text{otherwise.}
  \end{cases}
\]
Following an analogous argument as in Example~\ref{exa:heavydiamond}\,(a), we can easily describe the shortest 
paths connecting the points of $\{z\in \dOm\colon f(z)=t\}$ for $t\in(0,2)$:
\begin{itemize}
	\item If $t\le\frac{3}{4}$ or $t\ge\frac{11}{8}$, then $\partial E_t$ is a horizontal line segment.
  \item If $\frac34 < t \le t_0$, where $t_0 \approx 1.017$, then the shortest path is a piecewise affine line that passes 
  through the bottom tips of the large diamonds. The value of $t_0$ is found by equating the length 
  of the piecewise affine line that starts from a point on $\dOm$ whose $y$-value equals $t_0-1$ and passes through 
  the bottom tips of the large diamonds and the length of the piecewise affine 
  line that begins from the same point in $\dOm$ and passes
  through the top tips of all three diamonds.
  \item If $t_0 < t \le t_1$, where $t_1 \approx 1.127$, then the shortest path is a piecewise affine line that 
  begins from a point in $\dOm$ whose $y$-value is $t_1-1$, and passes
  through the top tips of the large diamonds and either the top or the bottom tip of the small diamond in the middle. 
  Such a possibility of choice of the tip for the small diamond is \emph{the cause of non-uniqueness of the solution}. The 
  value of $t_1$ is determined by finding the intersection of $\dOm$ with the ray connecting the top tip of the 
  small diamond in the middle and the top tip of one of the large diamonds.
  \item If $t_1 < t < \frac{11}{8}$, then the shortest path is a piecewise affine line that passes through the top tip of the small 
  diamond in the middle.
\end{itemize}

\noindent%
\begin{minipage}{\linewidth}
\vspace{2\topsep}
\center
\includegraphics[width=5.8in,height=2.4in,page=11]{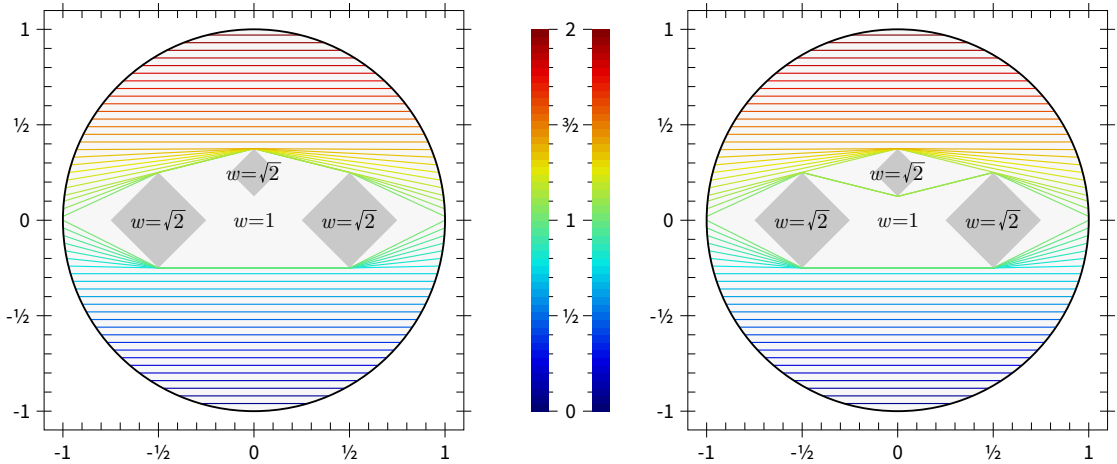}  

\vspace{\topsep}
\captionof{figure}{Level sets $\partial E_t$ for various values of $t \in (0,2)$ in Example~\ref{exa:3heavydiamonds}. 
In the
left figure, the shortest path for $t_0 < t \le t_1$ passes through the upper tip of the small diamond, and
passes through the lower tip in the right figure. Similarly as in Example~\ref{exa:heavydiamond}\,(a), the precise value 
of $\alpha \ge \sqrt2$ plays no role. 
\label{fig:3heavydiamonds}}
\end{minipage}

\noindent%
\begin{minipage}{\linewidth}
\center
\includegraphics[width=5.8in,height=2.4in,page=12]{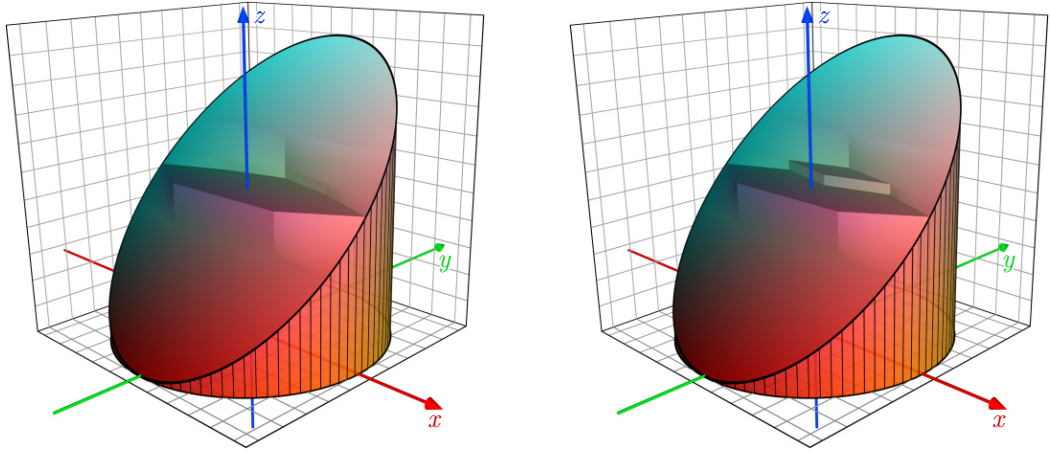}

\captionof{figure}{Graphs of the two distinct solutions in Example~\ref{exa:3heavydiamonds} 
whose level sets have been depicted in Figure~\ref{fig:3heavydiamonds}. 
\label{fig:3heavydiamonds3D}}
\vspace{\topsep}
\end{minipage}
Observe also that any convex combination of the solutions shown in 
Figures~\ref{fig:3heavydiamonds} and~\ref{fig:3heavydiamonds3D} is also a solution to the Dirichlet problem.
\end{exa}
Finally, the following example shows that the non-uniqueness of solutions is in general not caused by discontinuity of the weight function.
\begin{exa}\label{exa:litedmdheavycore}
Let the weight be given by
\[
  w(x,y) =
  \begin{cases}
    0.75 - 0.5(|x|+|y|) & \text{if } 0 \le |x|+|y| < 0.5, \\
    |x|+|y| & \text{if } 0.5 < |x|+|y| \le 1, \\
    1 &  \text{otherwise.}
  \end{cases}
\]
Let the three regions of $\Om$ listed above be denoted by 
$K_{\mathrm{in}}$, $K_{\mathrm{ann}}$, and $K_{\mathrm{out}}$ respectively.
Here again one can see that the boundary has positive mean curvature.

Suppose $t\ge 1$ and $\eta(x)=(x,g(x))$ is a curve that forms the boundary of $E_t$
(travelling from left to right).
In $A:=K_{\text{ann}}\cap \{x<0\}\cap \{y>0\}$, by Ibn Sahl--Snell law,
we know that $w\, \sin(\theta)=\kappa$ for some constant $\kappa$, where $\theta$ is the oriented angle
between the normal to the iso-line for the weight $w$ and the tangent vector to the curve $\eta$. Observe that 
this normal has slope $-1$. The value of
$\kappa$ might change from curve to curve, but for a given curve it is constant.
Since $\eta$ moves from left to right, it follows that
$\theta\in (-\pi/4,3\pi/4)$, considering that 
$\theta(\eta(x)) = \arctan(g'(x)) + \pi/4$. In particular, $\theta(\eta(x))$ is monotone increasing (resp. decreasing) if
and only if $g^\prime(x)$ is monotone increasing (resp. decreasing).
Moreover, the function $x\mapsto \theta(\eta(x))$ is continuous in the quadrant $\{x<0<y\} \cap \Om$ 
and smooth inside each of the regions $K_{\text{in}}$, $K_{\text{ann}}$, and $K_{\text{out}}$ within the quadrant. 

 Let us now discuss convexity/concavity of $g$ using the Ibn Sahl--Snell law, based on the value of $g'(x_0)$ at 
some point $\eta(x_0) \in A$. Since $\sin(\theta)= \kappa/w$, we can determine monotonicity of the function 
$x \mapsto \theta(\eta(x))$ and hence of $g'(x)$ in a neighborhood of $\eta(x_0)$ based on the monotonicity 
of $x \mapsto w(\eta(x))$ at $x_0$ and the sign of $\kappa$. Note however that one needs to pay special attention 
to and distinguish cases when $\theta = 0$ (since this is the borderline value, where the sign of $\kappa$ changes), and 
$\theta < \pi/2$ as opposed to $\theta > \pi/2$ (since the monotonicity of $x \mapsto w(\eta(x))$ is different in these two 
cases and so is the relation between monotonicity of $\sin(\theta)$ and $\theta$).

\begin{center}
\begin{tabular}{|c|c||c||c|c|c|c|}
\hline
\multicolumn{2}{|c||}{Value of} &  Sign of both & \multicolumn{4}{c|}{Monotonicity of}\\[1pt]
$g'(x_0)$ &  $\theta(\eta(x_0))$ & $\kappa$ and $\sin(\theta)$ & $w$ & $\sin(\theta)$ & $\theta$ & $g'$ \\[1pt]
\hline\hline
$(-\infty, -1)$ & $(-\pi/4, 0)$ & $-$ & decr. & decr. & decr. & decr. \\[1pt]
$-1$ & $0$ & $0$ & decr. & const. & const. & const. \\[1pt]
$(-1, 1)$ & $(0, \pi/2)$ & + & decr. & incr. & incr. & incr. \\[1pt]
$1$ & $\pi/2$ & + & const. & const. & const. & const \\[1pt]
$(1, \infty)$ & $(\pi/2, 3\pi/4)$ & + & incr. & decr. &  incr. & incr. \\[1pt]
\hline
\end{tabular}
\end{center}
From the table above, we can deduce that if $g^\prime(x_0) < -1$ at some point $x_0$ in $A$, then $g$ is 
concave within the entire region $A$. If $g^\prime(x_0) = -1$ at some point, then $g^\prime\equiv -1$ in $A$. 
Finally, if $g^\prime(x_0) > -1$ at some point in $A$, then $g$ is convex in all of $A$. 
The argument below showing that $\eta$ should intersect the $y$-axis horizontally also tells us that
the possibility $g^\prime(x)\le -1$ and the possiblity $g^\prime(x)\ge 1$ will not happen.

Analogous arguments can be made in $K_{\text{in}}$, leading us to conclude that
$\eta$ is either convex in $K_{\text{in}}$, or is concave in $K_{\text{in}}$. 
Similar arguments for regions in the other three quadrants yield analogous conclusions.
From this we deduce that $\eta$ has one-sided tangents where it intersects the $y$-axis. This will be
used to show below that $\eta$ will intersect the $y$-axis horizontally.

We now check that $\eta=\partial E_s$ intersects the $y$-axis horizontally 
for all $s\in (0, 2)$. Note that one-sided directional derivatives on either side of the $y$-axis 
exist, which can be seen from  concavity/convexity of the curve from the argument above.
Hence it suffices (up to small error terms) to compare weighted lengths of straight lines close to the $y$-axis. 
We compute the weighted length of the straight line path joining $(-\varepsilon, b)$ to the $y$-axis for some 
$-1<b<1$ and sufficiently small $\varepsilon >0$.
Actually, the following calculation discusses only the case when $0<b<0.5$, i.e., when 
$\partial E_s$ crosses the $y$-axis in $K_{\mathrm{in}}$ above the $x$-axis. Analogous computations 
can be done for other values of $b \in (-1, 1)$ when $\partial E_s$ crosses the $y$-axis somewhere
 in $K_{\mathrm{ann}}$, or in $K_{\mathrm{in}}$ below the $x$-axis. Consider such a 
 path $\varphi$ making an 
angle $\vartheta$, $-\pi/2<\vartheta<\pi/2$, with the horizontal. This takes the form
\[
\varphi(t)=(-\varepsilon+t\cos \vartheta, b+t\sin \vartheta), \qquad 0\leq t \leq \varepsilon/(\cos \vartheta).
\]
Clearly $w(\varphi(t))=0.75-0.5(\varepsilon-t\cos\vartheta + b + t\sin\vartheta)$ 
and $\|\varphi'(t)\|=1$ for all $t$. An easy calculation yields that the weighted length $I$ is given by
\[
4I=\frac{(3-2\varepsilon-2b)\varepsilon}{\cos \vartheta} + \frac{(\cos\vartheta - \sin\vartheta)\varepsilon^2}{\cos^{2}\vartheta}.
\]
From which we obtain
\[
4(\cos^{2}\vartheta)\frac{dI}{d\vartheta}=(3-2\varepsilon-2b)\varepsilon\sin\vartheta 
+ 2\varepsilon^{2}(\cos\vartheta - \sin\vartheta)\tan\vartheta-\varepsilon^{2}(\cos\vartheta+\sin\vartheta).
\]
For small $\varepsilon$ (compared to $\vartheta$) we see
\[
4(\cos^{2}\vartheta)\frac{dI}{d\vartheta} \approx (3-2b)\varepsilon \sin \vartheta
\]
which is negative for $\vartheta<0$ and positive for $\vartheta>0$. Hence one obtains a shorter 
length by making $\vartheta$ close to $0$. Making $\varepsilon$ smaller and smaller, we see that the weighted length
minimizing
curve $\eta$ will be horizontal where it intersects the $y$-axis.

Let us now verify that $\gamma = \partial E_1$ does not entirely coincide with the $x$-axis. 
We do so by comparing the weighted length of the line segment joining $(-0.5,0)$ to $(0.5,0)$
with the length of the curve that is the concatenation of the line segment joining $(-0.5,0)$ to
$(0,0.2)$ and the line segment joining $(0,0.2)$ to $(0.5,0)$. A direct computation shows that the first
path has length $5/8$, while the second path has length $0.575 \times \sqrt{1.16}$.
Thus, the second path is shorter (in the weighted sense) than the first path.
Since the first path forms part of the curve that coincides with the $x$-axis, the claim follows.

Suppose now that the superlevel set $E_1$ of $u$
is the minimal weak solution set for the boundary data $\chi_{F_1}$. 
Then $\gamma=\partial E_1$ intersects the $y$-axis at a point $(0, H)$ for some $0<H<0.5$,  and 
from the discussion above, we know
that it does so horizontally. 
Clearly, $0<H<0.5$ implies that $0.5<w(0,H)<0.75$.

We now claim that in the region $x>0$, $\gamma=\partial E_1$
first intersects the $x$-axis at a point $(a,0)$ with $0<a\le w(0,H)$. 
If not, then $\gamma$ will intersect the region
\[
R= \{(x,y)\in \Omega \colon x,y>0, x+y>w(0,H) \}.
\]
Clearly, $R\subset K_{\mathrm{ann}}\cup K_{\mathrm{out}}$. In $R\cap K_{\mathrm{ann}}$, we have $w(x,y)=x+y>w(0,H)$. 
Hence, there will be a point $(x_0,y_0)$ in the boundary of $R$ where $\gamma$ crosses into $R$, and at this
point $w(x_0,y_0)=x_0+y_0=w(0,H)$. Since $\gamma$ has slope zero at the point $(0,H)$ where the weight was $w(0,H)$,
it follows that $\gamma$ has slope zero at $(x_0,y_0)$ as well.
Due to convexity of $\gamma$ in $K_{\text{ann}}$, the slope of $\gamma$ would necessarily be strictly positive inside 
the region $R\cap K_{\mathrm{ann}}$. 
We also know that $\gamma$ will be a straight line segment in $R\cap K_{\mathrm{out}}$. It follows that $\gamma$ would never 
reach the point $(1,0)$, contradicting the definition of $\gamma$. This gives the claim. Hence $\gamma$ 
passes through a 
point $(a, 0)$ for some $0<a\le w(0,H)$ and then continues horizontally towards $(1,0)$.

Now suppose $t>1$ is such that $\eta=\partial E_{t}$ intersects the $y$-axis at a point $(0,\tilde{H})$ with 
$H < \tilde{H} < 0.5$. It follows 
from minimality of $E_{1}$, symmetry of the setting,
and the fact that $\eta$ cannot intersect $\gamma$ at the $y$-axis that $\eta$
is disjoint from $\gamma$. Hence, $\eta$ will never meet the $x$-axis, so its trajectory, outside the $y$-axis, is 
completely determined by Snell's law. Since $\eta$ is horizontal at $(0,\tilde{H})$, it 
follows from Snell's law that $\eta$ will be horizontal at a point of $K_{\mathrm{ann}}$ where the weight agrees with 
$w(0,\tilde{H})$, hence at a point $(x,y)$ with $0<x<0.75$ and $y>0$. If $\tilde{H}$ were greater 
than $0.5$, then $\eta$ would be horizontal exactly at one point, viz.\@ $(0, \tilde{H})$ on the $y$-axis. Using the arguments of 
Example~\ref{exa:lightdiamondtight}, $\eta$ intersects $\partial \Omega$ at a point whose height 
is bounded away from $0$. Consequently, $t>t_{0}$ for some $t_{0}>1$.

Now suppose $\eta=\partial E_{t}$ for some $1<t<t_{0}$. Then $\eta$ must intersect the $y$-axis at the point $(0,H)$. The 
trajectory of $\eta$ is determined by Snell's law until it meets the $x$-axis. Thus, $\eta=\partial E_{t}$ and 
$\gamma=\partial E_{1}$ coincide until they meet the $x$-axis. The curve $\eta$ then follows the $x$-axis for 
some time before leaving it and intersecting $\partial \Omega$ at a point with height $t-1$. This final trajectory is 
calculated as in Example~\ref{exa:lightdiamondtight}.

Thus, whether $E_1$ is a minimal solution set or not,
$\eta=\partial E_t$ has the following form for $t\in (2-t_0,t_0)$:
$\eta$ begins at the point in 
$\partial \Omega$ with $y$-coordinate $t-1$, then intersects the $x$-axis which it follows for some time. From there, 
$\eta$ can move either up (if $E_t$ is the minimal weak solution set) or down (if $E_t$ is the maximal weak solution set), intersecting 
the $y$-axis in a point of the form $(0,\pm H)$. 
To obtain $\eta$ in the region $x\leq 0$, simply reflect through the $y$-axis.

Note that if $E_{t}$ is minimal ($\partial E_{t}$ goes up) then $E_{s}$ must also be minimal for all $t_{0}\geq s\geq t$. 
Choosing $\partial E_{t_a}$ stay in the upper half-plane for some $t_a\in [2-t_0,t_0]$ 
yields a solution whose superlevel sets $E_t$ are minimal for  $t\in (t_a, 2)$. Choosing $\partial E_{t_b}$ to 
move into the lower half-plane
for some $t_b\in [2-t_0,t_0]$ leads to another solution, 
whose superlevel sets $E_t$ are maximal for all $t\in (0,t_b)$. Contour curves of two distinct such solutions,
the first corresponding to $t_a=2-t_0$, and the second corresponding to the choice of $t_b=t_0$, are
shown below.
A more careful analysis using the equation $\sin(\arctan(g^\prime(x))+\pi/4)=\kappa/w(x,g(x))$ gave us the following picture of the two solutions.
The left figure in the two pictures is of the contour curves of the solutions, while the right figure
is the height-map of the solutions. In the first picture, the solution takes on the value $2-t_0$ in the middle lenticular region,
while in the second picture, the solution takes on the value $t_0$ there.

\noindent%
\begin{minipage}{\linewidth}
\vspace{\topsep}
\center
\includegraphics[width=5.8in,height=2.4in,page=13]{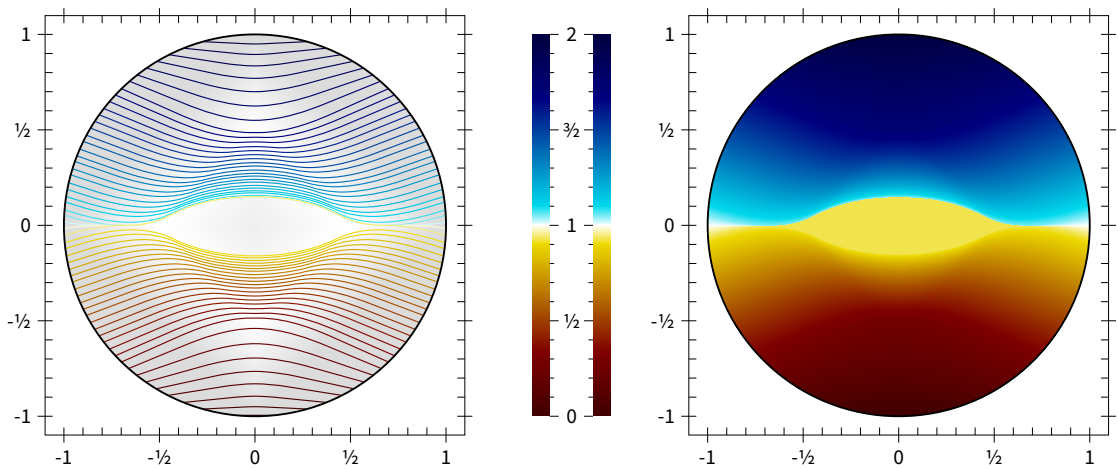}

\captionof{figure}{{\small Level sets $\partial E_t$ for various values of $t \in (0,2)$ in Example~\ref{exa:litedmdheavycore}, when 
the solution is constructed using minimal weak solution sets as the ``superlevel pancakes''. A heightmap of the solution 
$u$ is shown on the right; the color at each point represents the solution's function value at that 
point.} \label{fig:litedmdheavycoreA}}
\end{minipage}

\vskip .5cm

\noindent%
\begin{minipage}{\linewidth}
\vspace{\topsep}
\center
\includegraphics[width=5.8in,height=2.4in,page=14]{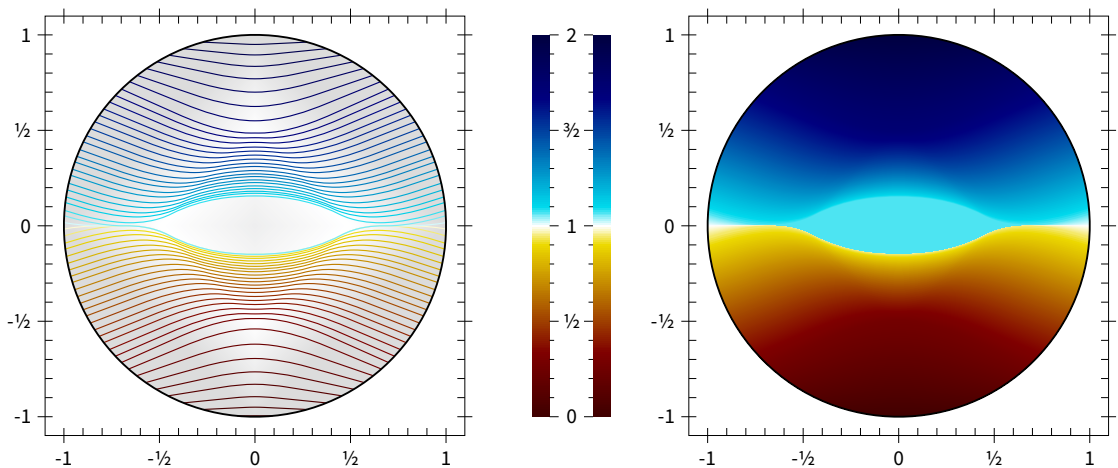}

\captionof{figure}{{\small Level sets and a heightmap of the solution in Example~\ref{exa:litedmdheavycore} that 
is constructed using maximal weak solution sets as the ``superlevel pancakes''. Note that 
$u(x,y) = 2 - \tilde{u}(x,-y)$, where $\tilde{u}$ is the solution shown 
in Figure~\ref{fig:litedmdheavycoreA}.} \label{fig:litedmdheavycoreB}}
\end{minipage}
\end{exa}

\bigskip
{\footnotesize
\noindent P.\,L., L.\,M., N.\,S., G.\,S.: Department of Mathematical Sciences, P.O. Box 210025, University of Cincinnati, Cincinnati,~OH~45221-0025, U.S.A.}

{\footnotesize
\noindent E-mail: {\tt panu.lahti@aalto.fi}, {\tt lukas.maly@uc.edu}, {\tt shanmun@uc.edu}, {\tt gareth.speight@uc.edu}
}
\end{document}